\documentclass[a4paper,12pt,leqno]{article}
\usepackage[latin1]{inputenc}
\usepackage{Macro-STC}
\usepackage{amsmath}
\usepackage{amsthm}
\usepackage{amsfonts}
\usepackage{amssymb}
\usepackage[cmtip,all]{xy}
\usepackage{mathrsfs}
\usepackage{stmaryrd}
\usepackage{bbm}
\usepackage{oldgerm}

\theoremstyle{plain}
\newtheorem{Lemma}{Lemma}
\newtheorem{Thm}[Lemma]{Theorem}
\newtheorem{Cor}[Lemma]{Corollary}
\newtheorem{Prop}[Lemma]{Proposition}
\theoremstyle{definition}
\newtheorem{Example}[Lemma]{Example}
\newtheorem{Notation}[Lemma]{Notation}
\newtheorem{Defn}[Lemma]{Definition}
\theoremstyle{remark}
\newtheorem{Remark}[Lemma]{Remark}

\numberwithin{Lemma}{section}
\numberwithin{equation}{section}

\DeclareMathOperator{\uIsom}{\underline{Isom}}

\DeclareMathOperator{\uHom}{\underline{Hom}}
\DeclareMathOperator{\uH}{\underline{H}}

\DeclareMathOperator{\Image}{Im}
\DeclareMathOperator{\Ext}{Ext}
\DeclareMathOperator{\interior}{int}
\newcommand{\ww}{\mathbbm w}

\begin{document}

\title{Stratifications of Newton polygon strata and Traverso's conjectures for $p$-divisible groups}

\author{Eike Lau\footnote{elau@math.upb.de}, Marc-Hubert Nicole\footnote{nicole@iml.univ-mrs.fr} , Adrian
Vasiu\footnote{adrian@math.binghamton.edu}}

\maketitle

\centerline{\it dedicated to Thomas Zink, for his 60th anniversary}

\bigskip\medskip

\noindent{\scshape Abstract.\ } The isomorphism number (resp.\
isogeny cutoff) of a $p$-divisible group $D$ over an algebraically
closed field of characteristic $p$ is the least positive integer $m$ such that $D[p^m]$
determines $D$ up to isomorphism (resp.\ up to isogeny). We show
that these invariants are lower semicontinuous in families of
$p$-divisible groups of constant Newton polygon.
Thus they allow refinements of Newton polygon
strata. In each isogeny class of $p$-divisible groups, we
determine the maximal value of isogeny cutoffs and give an upper
bound for isomorphism numbers, which is shown to be optimal in the
isoclinic case. In particular, the latter
disproves a conjecture of Traverso. As an application, we answer a question of Zink on the liftability of an endomorphism of $D[p^m]$ to $D$.




\section{Introduction}

Let $k$ be an algebraically closed field of positive
characteristic $p$. Let $D$ be a $p$-divisible group over $k$.
It is well-known that $D$ is determined by some
finite truncation $D[p^m]$ of sufficiently large level $m$.
This allows to associate to $D$ two numerical invariants:
The {\em isomorphism number} $n_D$ is the least
level $m$ such that $D[p^m]$ determines $D$ up
to isomorphism, and the {\em isogeny cutoff}\/ $b_D$ is
the least level $m$ such that $D[p^m]$ determines
$D$ up to isogeny.
\footnote{This differs from the definitions in \cite{NV2}
and \cite{Va3}, which give $n_D=b_D=0$ if either $D$ or its 
dual is \'etale, while the definitions we use here give
always $n_D,b_D\geq 1$.}

In this paper we study how these invariants of $D$ behave
in families and how large they can get. As it turns out,
the following distance function on isogeny classes of
$p$-divisible groups has closely related properties. The {\it distance} $q_{D,E}$ between two $p$-divisible groups $D$ and $E$ over $k$ is the minimal non-negative integer $m$ such that
there exists an isogeny $D\to E$ with kernel annihilated by $p^m$,
while $q_{D,E}=\infty$ if no such $m$ exists. The {\it minimal height}
of $D$ is defined to be $q_D=q_{D,D_0}$ where $D_0$ is the unique
(up to isomorphism) minimal $p$-divisible group in the isogeny class of $D$. We recall
that a minimal $p$-divisible group is characterized by its isomorphism number being $1$. 

The numbers $b_D$, $n_D$, $q_D$, and $q_{D,E}$ are invariant under extensions of the algebraically closed base field $k$; see Lemma \ref{Le-perm} and Corollary \ref{Co-perm}.
For a $p$-divisible group $\Delta$ over an arbitrary field $\kappa$ of characteristic $p$ we define $b_\Delta=b_{\Delta_{\bar\kappa}}$, $n_\Delta=n_{\Delta_{\bar\kappa}}$, etc., where $\bar\kappa$ is an algebraic closure of $\kappa$.


\subsection*{Families}

If $\scrD$ is a $p$-divisible group over an $\FF_p$-scheme $S$,
for each point $s\in S$ we denote by $b_\scrD(s)$, $n_\scrD(s)$,
and $q_\scrD(s)$ the isogeny cutoff, the isomorphism number, and
the minimal height (respectively) of the geometric fibre $\scrD_{\bar s}$ of $\scrD$ over $s$. If $\scrE$ is another $p$-divisible group over $S$ we write
$q_{\scrD,\scrE}(s)=q_{\scrD_{\bar s},\scrE_{\bar s}}$. 

\begin{Thm}
\label{Th-semicont}
Let $\scrD$ and $\scrE$ be two $p$-divisible groups with constant
Newton polygon over an $\FF_p$-scheme $S$ and let $m$ be
a non-negative integer. 

\smallskip
(a) The set $U_{b_\scrD}=\{s\in S\mid b_\scrD(s)\leq m\}$ is
closed in $S$.

\smallskip
(b) The set $U_{n_\scrD}=\{s\in S\mid n_\scrD(s)\leq m\}$ is
closed in $S$.

\smallskip
(c) The set $U_{q_\scrD}=\{s\in S\mid q_\scrD(s)\leq m\}$ is
closed in $S$.

\smallskip
(d) The set $U_{q_{\scrD, \scrE}}=\{s\in S\mid
q_{\scrD,\scrE}(s)\leq m\}$ is closed in $S$.
\end{Thm}

\noindent 
It follows that $S$ carries three natural
stratifications into a finite number of reduced locally closed
subschemes associated to $\scrD$: The strata of the {$b$-stratification} are the loci where the function
$b_\scrD:S\to\NN$ is constant; the {$n$-stratification} and the {$q$-stratification} are defined similarly.

Let us repeat in words part (d) for $m=0$: {\it The set of points
of $S$ over which the geometric fibres of $\scrD$ and $\scrE$ are
isomorphic, is a closed subset of $S$}. When either $\scrD$ or $\scrE$ is a constant $p$-divisible group, this is \cite[Thm.~2.2]{Oo2}.

The main step in the proof of Theorem \ref{Th-semicont} is the
following going down principle: {\em For $p$-divisible groups
over $k[[t]]$ with constant Newton polygon, the numbers $b_D$,
$n_D$, $q_D$, and $q_{D,E}$ go down under specialization}
(see Theorems \ref{Th-going-down-q}, \ref{Th-going-down-b},
\ref{Th-going-down-n}, and Corollary \ref{Co-going-down-q}).
To deduce the general case of Theorem \ref{Th-semicont} we pass through a truncated variant
of Theorem \ref{Th-semicont} presented in Theorems \ref{Th-semicont-trunc} and \ref{Th-semicont-trunc-q}; this allows us to use
the fact that truncated Barsotti--Tate groups have parameter spaces
(locally) of finite type. Theorem \ref{Th-semicont} is proved at the end of Subsection \ref{Se-pf-semicont}.

\subsection*{Explicit upper bounds}

In the following we fix a non-trivial $p$-divisible group $D$ 
over $k$ of dimension $d$, codimension $c$, and Newton polygon
$\nu:[0,c+d]\to\RR$. 
We are interested in upper bounds of $b_D$
and $n_D$ in terms of either $\nu$ or $c$ and $d$. We recall that
the $a$-number of $D$ is $a_D=\dim_k(\Hom(\pmb{\alpha}_p,D))$. 
Let us begin with the isogeny cutoff $b_D$ and let
$$
j(\nu)=
\begin{cases}
\nu(c)+1 & \text{if $(c,\nu(c))$ is a breakpoint of $\nu$,} \\
\lceil\nu(c)\rceil & \text{otherwise.}
\end{cases}
$$
Here endpoints of $\nu$ are considered as breakpoints.

\begin{Thm}
\label{Th-b-maximal}
We have $b_D\leq j(\nu)$ with equality when $a_D\leq 1$.
\end{Thm}

\noindent 
This is proved in Section \ref{Se-val-isog}. Note that
$\nu(c)\leq cd/(c+d)$ with equality precisely when $\nu$ is
linear. Thus Theorem \ref{Th-b-maximal} refines the {Traverso
isogeny conjecture} which is proved in \cite{NV2} and which asserts that $b_D\leq\lceil cd/(c+d)\rceil$ when $cd>0$ with equality for some $D$ of dimension $d$ and codimension $c$.

We have a similar upper bound for isomorphism numbers.

\begin{Thm}
\label{Th-n-upper}
If $D$ is not ordinary, then $n_D\leq\lfloor 2\nu(c)\rfloor$.
\end{Thm}

\noindent 
If $D$ is ordinary then $n_D=1$ and $\nu(c)=0$.
We refer to Corollary \ref{Co-fD-upper} for the proof of Theorem \ref{Th-n-upper}.
We expect that this upper bound of $n_D$ is optimal for every
Newton polygon
$\nu$, but in this paper this is proved only if $\nu$ is linear
i.e., if $D$ is isoclinic (see Proposition \ref{Pr-n-isocl}).
We thus conclude that:

\begin{Cor}
\label{Co-n-upper}
Assume that $cd>0$.
Then $n_D\leq\lfloor 2cd/(c+d)\rfloor$
with equality for some (isoclinic) $p$-divisible
group $D$ of dimension $d$ and codimension $c$.
\end{Cor}

Our search for upper bounds of $n_D$ was guided by the Traverso
truncation conjecture \cite[\S 40, Conj.~4]{Tr3} which predicts
that $n_D\leq\min\{c,d\}$ if $cd>0$. This estimate is 
well-known if $\min\{c,d\}=1$. It is verified
for supersingular $p$-divisible groups in \cite[Thm.~1.2]
{NV1} and for quasi-special $p$-divisible groups in \cite[Thm.~1.5.2]{Va3}; for $|d-c|\le 2$ it follows indeed from Corollary \ref{Co-n-upper}.
But Corollary \ref{Co-n-upper} also shows that the original 
conjecture is
wrong in general, even for isoclinic $p$-divisible groups; the first counterexamples show up when $\{c,d\}=\{2,6\}$.
For a fixed positive value of $t=\min\{c,d\}$,
the natural number $\lfloor 2cd/(c+d)\rfloor$ can be any integer
in the interval $[t,2t-1]$.

Quantitative upper bounds of $n_D$ in terms of $c$ and $d$ can be
traced back to \cite[Thm.~3]{Tr1}, where the inequality $n_D\leq
cd+1$ is established. A weaker upper bound with a simpler proof can be
found in \cite[Thm.~1]{Tr2}. More recently, \cite[Cor.~1]{GV} shows
that $n_D\le cd$ if $cd>0$ and that $n_D\le cd+1-a_D^2$ if $D$ is not ordinary.

A key result of \cite{GV} characterizes $n_D$ as the minimal positive integer such that the truncation homomorphism 
$\End(D[p^{n_D+1}])\to\End(D[p])$ has finite image, cf. \cite[Cor.~2 (b)]{GV}.
Using this characterization, we prove Theorem \ref{Th-n-upper}
in the case $a_D=1$ by a detailed analysis of Dieudonn\'e
modules over $k$. The general case of Theorem \ref{Th-n-upper} follows
by the going down principle, using the fact that $D$ is the specialization
of a $p$-divisible group over $k((t))$ with $a$-number at
most one and with Newton polygon $\nu$ by \cite
{Oo1}.

Assume that $D$ is equipped with a principal quasi-polarisation 
$\lambda$; thus $c=d>0$. The isomorphism number $n_{D,\lambda}$ 
is the least level $m$ such that $(D[p^m],\lambda[p^m])$ 
determines $(D,\lambda)$ up to isomorphism. Based on
\cite{GV}, \cite{NV1}, and Theorem \ref{Th-n-upper} we prove in Subsection
\ref{Se-p.q.p.} that $n_{D,\lambda}\le d$;
this bound is optimal.

\subsection*{Relation with minimal heights}

Another approach to bound $n_D$ and $b_D$ from above is based on
the fact that an isogeny of $D$ with kernel annihilated by $p$
changes $b_D$ at most by one and $n_D$ at most by two, see Lemma
\ref{Le-b-change} and Proposition \ref{Pr-f-change}. It turns 
out that the results on Dieudonn\'e modules used in the proof 
of Theorem \ref{Th-n-upper}
also give the following upper bound of minimal heights (see
end of Subsection \ref{Subsec-min-mod}):

\begin{Thm}
\label{Th-min-isog}
We have $q_D\leq\lfloor\nu(c)\rfloor$ with equality when $a_D=1$.
In other words, if $D_0$ is the unique (up to isomorphism)
minimal $p$-divisible group over $k$ of Newton polygon $\nu$,
then there exists an isogeny $D\to D_0$ whose kernel is
annihilated by $p^{\lfloor\nu(c)\rfloor}$, and this exponent
is optimal when $a_D=1$.
\end{Thm}

By the new approach, this result gives the inequalities 
$b_D\leq 1+\lfloor\nu(c)\rfloor$ and $n_D\leq 1+2\lfloor\nu(c)\rfloor$ because $b_{D_0}=n_{D_0}=1$.
The first estimate coincides
with the upper bound in Theorem \ref{Th-b-maximal}
except when $\nu(c)$ is an integer and $\nu$ is linear at $c$,
in which case it is off by $1$. 
When $\lfloor 2\nu(c)\rfloor$ is odd, the second estimate is precisely 
Theorem \ref{Th-n-upper}, while it is again off by $1$ otherwise. 

We remark that the existence of some upper bound of $q_D$ is
already proved in \cite[p.~44]{Ma}; see also \cite[p.~270]{Oo2}.
The supersingular case of Theorem \ref{Th-min-isog} follows from \cite[Rmk.~2.6 and Cor.~3.2]{NV1}.

\subsection*{Lifting of endomorphisms}

We apply the explicit upper bounds of $n_D$ to the lifting
of endomorphisms of truncations of $D$ and to the 
level torsion $\ell_D$ of $D$ defined in \cite{Va3}.

There exists a non-negative integer $e_D$,
which we call the \emph{endomorphism number} of $D$,
characterized by the following property:
For positive integers $m\geq n$, the two restriction homomorphisms 
$$
\End(D)\xrightarrow{\tau_{\infty,n}}
\End(D[p^n])\xleftarrow{\tau_{m,n}}\End(D[p^m])
$$
have equal images if and only if $m\geq n+e_D$. In other words:
An endomorphism of $D[p^n]$ lifts to an endomorphism of $D$
if and only if it lifts to an endomorphism of $D[p^{n+e_D}]$,
and $e_D$ is minimal with this property for each positive 
integer $n$ separately; see Lemma \ref{Le-eDE}.

Let $(M,F,V)$ be the covariant Dieudonn\'e module of $D$.
In Subsection \ref{onell} we recall the definition of the \emph{level module} 
$O\subseteq\End_{W(k)}(M)$ introduced in \cite{Va3}. The \emph{level torsion} $\ell_D$ is
the smallest non-negative integer such that $p^{\ell_D}\End_{W(k)}(M)\subseteq O$.
\footnote{This differs from the definition in \cite{Va3},
which gives $\ell_D=1$ if $D$ is ordinary and $cd>0$, while the definition 
we use here gives $\ell_D=0$ when $D$ is ordinary.}

These new invariants are related to $n_D$ 
as follows (see Subsection \ref{Se-conclusions}):

\begin{Thm}
\label{Th-n-ell-e}
If $D$ is a non-ordinary $p$-divisible group over an algebraically
closed field $k$, then we have \
$
n_D=\ell_D=e_D.
$
\end{Thm}

\noindent
If $D$ is ordinary we have $n_D=1$ and $\ell_D=e_D=0$. 
We assume now that $D$ is not ordinary.
In \cite{Va3} it was shown that $n_D\le\ell_D$ and that
the equality holds if $D$ is a direct sum of isoclinic
$p$-divisible groups; the equality was also expected to hold
in general therein. In this paper, we prove the inequalities
$$
n_D\leq e_D\leq\ell_D\leq n_D.
$$
The second inequality $e_D\leq\ell_D$ is not too difficult.
The other two inequalities use again the \cite{GV} characterization of 
$n_D$ mentioned above.
Then $n_D\leq e_D$ is immediate, but the inequality $\ell_D\leq n_D$ is a lot more involved.

Together with the upper bound of Theorem \ref{Th-n-upper}
we obtain the following effective lifting of endomorphisms,
which answers a question of Th.~Zink.

\begin{Cor}
\label{Co-end-lift}
Let $\nu$ and $c$ be the Newton polygon and the codimension of $D$. Let $n$ be a positive integer.
An endomorphism of $D[p^n]$ lifts to an endomorphism of $D$
if and only if it lifts to an endomorphism of 
$D[p^{n+\lfloor 2\nu(c)\rfloor}]$.
\end{Cor}

For similar results on homomorphisms we refer to Subsections \ref{Se-conclusions} and  \ref{Up-b}. As a special case, for each $h\in\NN$ 
we compute the minimal number $N_h$ such that 
for every two $p$-divisible groups
$D$ and $E$ over $k$ of height at most $h$, a homomorphism 
$D[p^n]\to E[p^n]$ lifts to a homomorphism $D\to E$ if and
only if it lifts to a homomorphism $D[p^{n+N_h}]\to E[p^{n+N_h}]$: By Proposition \ref{Pr-Nh} we have $N_h=\lfloor h/2\rfloor$. 

\bigskip
{\it Terminology.}
A {\em $BT$ group of level $n$} is a truncated Barsotti--Tate group 
of level $n$. 
We denote by $\NN^*$ the set of positive integers.


\section{Preliminaries}
\label{Se-prelim}

We begin with a lemma on homomorphisms. 

\begin{Lemma}
\label{Le-eDE}
Let $D$ and $E$ be $p$-divisible groups over $k$.

\smallskip
(a) For each positive integer $n$ there exists a non-negative
integer $e_{D,E}(n)$ with the following property:
For $e\in\NN$, the two restriction maps 
$$
\Hom(D,E)\to\Hom(D[p^{n}],E[p^{n}])
\leftarrow\Hom(D[p^{n+e}],E[p^{n+e}])
$$ 
have equal images if and only if $e\ge e_{D,E}(n)$.

\smallskip
(b) There exists an upper bound of $e_{D,E}(n)$  
in terms of the heights of $D$ and $E$.

\smallskip
(c) The number $e_{D,E}(n)$ does not depend on $n$.
\end{Lemma}

\begin{proof}
For (a) and (b) we refer to \cite[Prop.~1.6]{Oo2} or \cite[Thm.~5.1.1 (c)]{Va1}. We prove (c). Let $H_{\infty}=\Hom(D,E)$. 
For $n\in\NN^*$ let $H_n=\Hom(D[p^n],E[p^n])$. 
Following \cite[Subsect.~2.1]{GV}, we have two exact sequences
$$
0 \to H_\infty \xrightarrow p H_\infty \to H_1, 
\qquad
0 \to H_n \xrightarrow{\iota_n} H_{n+1} \to H_1, 
$$
where $\iota_n$ maps $u\in H_n$ to the obvious composition 
$$
D[p^{n+1}]\to D[p^{n}]\xrightarrow uE[p^{n}]\to E[p^{n+1}].
$$
For $m\in\NN\cup\{\infty\}$ with $m\ge n$ 
let $H_{m,n}$ be the image of $H_m\to H_n$.
One deduces that for all $e\in\NN$ we have 
a homomorphism of exact sequences
with vertical injections:
\begin{equation}
\label{Eq-eED}
\xymatrix@M+0.2em{
0 \ar[r] & H_{\infty,n} \ar[r]^-{\iota_n} \ar[d] &
H_{\infty,n+1} \ar[r] \ar[d] & H_{\infty,1} \ar[d] \ar[r] & 0 \\
0 \ar[r] & H_{n+e,n} \ar[r]^-{\iota_n} & 
H_{n+1+e,n+1} \ar[r] & H_{1+e,1}.
}
\end{equation}
The snake lemma implies that $e_{D,E}(n)\leq e_{D,E}(n+1)\leq\max\{e_{D,E}(n),e_{D,E}(1)\}$. By induction on $n\in\NN^*$ we get that $e_{D,E}(n)=e_{D,E}(1)$. Thus (c) holds.
\end{proof}

Lemma \ref{Le-eDE} (c) allows to define:

\begin{Defn}
The \emph{homomorphism number} $e_{D,E}\in\NN$ of $D$ and $E$
is the constant value of $e_{D,E}(n)$ for $n\in\NN^*$. 
The \emph{endomorphism number} of $D$ is defined as $e_D=e_{D,D}$.  
\end{Defn}

By Lemma \ref{Le-eDE} (b), 
for $h\in\NN$ there exists a minimal number $N_h\in\NN$ 
such that for every two $p$-divisible groups $D$ and $E$ over $k$ of height 
at most $h$ we have $e_{D,E}\leq N_h$. In Proposition \ref{Pr-Nh},
we will see that $N_h=\lfloor h/2\rfloor$.

We continue with a formal definition of the numerical
invariants $n_D$ and $b_D$.

\begin{Defn}
\label{Def-main} Let $D$ be a $p$-divisible group over $k$. The
{\em isomorphism number} $n_D$ (resp.\ the {\em isogeny cutoff}
$b_D$) is the smallest non-negative integer $m$ with the following
property: If $E$ is a $p$-divisible group over $k$ such that
$E[p^m]$ is isomorphic to $D[p^m]$, then $E$ is isomorphic (resp.\
isogenous) to $D$.
\end{Defn}

\noindent 
Such integers $m$ exist, for example $m=1+N_h$ if $D$ 
has height $h$. Thus $b_D$ and $n_D$ are well-defined. 
The following properties are easily established.

\begin{Lemma}
\label{Le-prelim} For a  $p$-divisible group $D$ over $k$ let
$D^\vee$ be the dual of $D$ and let $D=D^{\text{\rm
\'et}}\times D^{\circ}$ be the canonical product decomposition, where
$D^{\text{\rm \'et}}$ is \'etale and $D^{\circ}$ is connected.
We have:

\smallskip
(a) {} $1\leq b_D\leq n_D$;

\smallskip
(b) {} $b_D=b_{D^\vee}\,$ and $\:n_D=n_{D^\vee}$;

\smallskip
(c) {} $b_D=b_{D^{\circ}}\,$ and $\:n_D=n_{D^{\circ}}$.
\end{Lemma}

\noindent
In particular, if $D$ is ordinary, then $b_D=n_D=1$. Similarly, if
$\min\{c,d\}=1$, then $b_D=n_D=1$ because all one-dimensional connected $p$-divisible groups of given height are isomorphic (see \cite[Ch.~IV, Sect.~8, Prop.]{De}). Thus the invariants $b_D$ and $n_D$
are interesting only when $\min\{c,d\}\ge 2$.

\medskip

We consider the following distance function on isogeny classes.

\begin{Defn}
Let $D$ and $E$ be two $p$-divisible groups over $k$. If $D$ and $E$
are isogenous, then their distance $q_{D,E}$ is the smallest
non-negative integer $m$ such that there exists an isogeny $\rho:D\to
E$ with $\Ker(\rho)\subseteq D[p^m]$. If $D$ and $E$ are not
isogenous we define $q_{D,E}=\infty$.
\end{Defn}

Again, the next lemma is easily checked.

\begin{Lemma}
\label{Le-q} 
The following three properties hold:

\smallskip
(a) {} We have $q_{D,E}=0$ if and only if $D$ and $E$ are isomorphic.

\smallskip
(b) {} We have $q_{D,E}=q_{E,D}=q_{D^\vee,E^\vee}$.

\smallskip
(c) {} If $q_{D,E}<\infty$, then $q_{D,E}=q_{D^\circ,E^\circ}$.
\end{Lemma}

Recall that a $p$-divisible group $D$ over $k$ is called
\emph{minimal} if $\End(D)$ is a maximal order in
$\End(D)\otimes_\ZZ\QQ$; this is equivalent to the condition $n_D=1$.
For each $p$-divisible group $D$ over $k$, there exists a unique (up to isomorphism) minimal
$p$-divisible group $D_0$ over $k$ isogenous to $D$. See \cite[Subsects.~1.1-1.2]{Oo3} and \cite[Thm.~1.6]{Va3} for these facts. Thus following \cite{NV2} we define:

\begin{Defn}
The \emph{minimal height} of $D$ is $q_D=q_{D,D_0}$.
\end{Defn}

\noindent
As $(D_0)^\vee$ and $(D_0)^\circ$ are also minimal,
Lemma \ref{Le-q} implies that
$$
q_D=q_{D^\vee}=q_{D^\circ}.
$$

We have the following permanence properties.

\begin{Lemma}
\label{Le-perm}
Let $k\subseteq \kappa$ be an extension of
algebraically closed fields. For
$p$-divisible groups $D$ and $E$ over $k$ we have:

\medskip
(a) {} $\Hom(D,E)\cong\Hom(D_{\kappa},E_{\kappa})$;

\smallskip
(b) {} $q_{D,E}=q_{D_\kappa,E_\kappa}$ and $q_D=q_{D_\kappa}$;

\smallskip
(c) {} $e_{D,E}=e_{D_\kappa,E_\kappa}$.
\end{Lemma}

\begin{proof}
Part (a) is well-known, and (b) follows from (a). 
For positive integers $m\geq n$, let ${\uH}{}_{m,n}$ be the scheme theoretic image of the reduction homomorphism $\uHom(D[p^m],E[p^m])\to\uHom(D[p^n],E[p^n])$.
If $l\ge m$ is an integer, then ${\uH}{}_{l,n}$ is a subgroup scheme of ${\uH}{}_{m,n}$.
We have $m-n\ge e_{D,E}$ if and only if ${\uH}{}_{m,n}(k)={\uH}{}_{l,n}(k)$
for all integers $l\geq m$. As ${\uH}{}_{m,n}$ is of finite type over $k$ and its definition is compatible with the base change from $k$ to $\kappa$, we get (c).
\end{proof}

\noindent
The identities $n_D=n_{D_{\kappa}}$ and $b_D=b_{D_{\kappa}}$ also hold, cf.\ Corollary \ref{Co-perm} below.


\section{The going down principle}

In this section, we prove Theorem \ref{Th-semicont} in the case
$S=\Spec k[[t]]$. By standard arguments this implies that the
functions $n_\scrD$, $b_\scrD$, $q_\scrD$, and $q_{\scrD,\scrE}$ go
down under specialization when the base scheme is noetherian.

\subsection{Distances and minimal heights}

Going down of distance numbers follows from
the constancy results of \cite{OZ}:

\begin{Thm}
\label{Th-going-down-q}
Let $R=k[[t]]$, let $K=k((t))$,
and let $\bar K$ be an algebraic closure of $K$.
If $\scrD$ and $\scrE$ are $p$-divisible groups over
$R$ with constant Newton polygon $\nu$, then
$$
q_{\scrD_k,\scrE_k}\leq q_{\scrD_{\bar K},\scrE_{\bar K}}.
$$
\end{Thm}

\begin{proof}
Let $\bar R$ be the normalization of $R$ in $\bar K$.
By \cite[Cor.~3.2]{OZ} there exist $p$-divisible groups
$D$ and $E$ over $k$ and isogenies
$$
\rho:D_{\bar R}\to\scrD_{\bar R}
\qquad\text{and}\qquad
\omega:E_{\bar R}\to\scrE_{\bar R}.
$$
Let $m=q_{\scrD_{\bar K},\scrE_{\bar K}}$. Note that $m<\infty$.
By definition there exists an isogeny
$\xi_{\bar K}:\scrD_{\bar K}\to\scrE_{\bar K}$
with kernel annihilated by $p^m$.
As $\Hom(D,E)$ is equal to $\Hom(D_{\bar K},E_{\bar K})$,
the composite quasi-isogeny over ${\bar K}$
$$
\chi_{\bar K}:D_{\bar K}\xrightarrow{\rho_{\bar K}}
\scrD_{\bar K}\xrightarrow{\xi_{\bar K}}\scrE_{\bar K}
\xleftarrow{\omega_{\bar K}}E_{\bar K}
$$
is defined over $k$ i.e., it arises from a quasi-isogeny
$\chi:D\to E$. Hence $\xi_{\bar K}$ extends to a quasi-isogeny
over ${\bar R}$
$$
\xi:\scrD_{\bar R}\xleftarrow{\rho}D_{\bar R}\xrightarrow{\chi_{\bar R}}
E_{\bar R}\xrightarrow{\omega}\scrE_{\bar R}.
$$
As $\xi$ and $p^m\xi^{-1}$ are isogenies over $\bar K$,
by the well-known Lemma \ref{Le-isog-closed} below the same is true over $\bar R$. As the residue field of the local ring $\bar R$
is $k$, the special fibre of $\xi$ is an isogeny
$\xi_k:\scrD_k\to\scrE_k$ with kernel annihilated by $p^m$,
which implies that $q_{\scrD_k,\scrE_k}\leq m$ as
required.
\end{proof}

\begin{Lemma}[{\cite[Prop.~2.9]{RZ}}]
\label{Le-isog-closed}
For a quasi-isogeny $\chi:\scrD\to\scrE$ of $p$-divisible
groups over a scheme $T$ of characteristic $p$, there exists a
unique closed subscheme $T_0\subseteq T$ such that a morphism
$T'\to T$ factors through $T_0$ if and only if
$\chi_{T'}$ is an isogeny.
\end{Lemma}

\begin{Cor}
\label{Co-going-down-q}
With $R$ and $\bar K$ as in Theorem \ref{Th-going-down-q}, for a $p$-divisible group $\scrD$ over $R$
with constant Newton polygon $\nu$ we have
$$
q_{\scrD_k}\leq q_{\scrD_{\bar K}}.
$$
\end{Cor}

\begin{proof}
Let $D_0$ be the minimal $p$-divisible group over $k$
with Newton polygon $\nu$ and apply Theorem \ref{Th-going-down-q}
with $\scrE=D_0\times_{\Spec k}\Spec R$.
\end{proof}

\begin{Remark}
The proof of Theorem \ref{Th-going-down-q} could also be
formulated in terms of Dieudonn\'e modules over the
perfect ring $\bar R$. Then the reference to \cite{OZ}
can be replaced by \cite[Thm.~2.7.4]{Ka}. Here one only
needs to know that to a $p$-divisible group over $\bar R$
one can associate a covariant Dieudonn\'e module (by \cite[Cor.~3.4.3]{Be} 
this association is  actually an equivalence of categories).
 \end{Remark}

\subsection{An auxiliary family of $p$-divisible groups}

The following general construction is used in the proof of the going down 
principle for isogeny cutoffs, but we prove more than is actually needed.

\begin{Lemma}
\label{Le-aux-fam}
Let $\scrD$ be an infinitesimal $p$-divisible group
of height $h$ over an affine $\FF_p$-scheme $X$
and let $m\geq 1$ be an integer.
Then there exists a vector bundle $Y\to X$ of rank $h^2$ and
an infinitesimal $p$-divisible group $\scrE$ over $Y$
such that the following three properties hold:

\smallskip
(i) If $o:X\to Y$ denotes the zero section,
then we have $o^*\scrE\cong\scrD$.

\smallskip
(ii) If $\pi:Y\to X$ denotes the projection,
then we have $\scrE[p^m]\cong\pi^*\scrD[p^m]$.

\smallskip
(iii) For each geometric point $x:\Spec \kappa\to X$,
every $BT$ group $B$ of level $m+1$ over $\kappa$
with $B[p^m]\cong x^*\scrD[p^m]$ is isomorphic
to $y^*\scrE[p^{m+1}]$ for some geometric point
$y:\Spec \kappa\to Y$ with $\pi\circ y=x$.
\end{Lemma}

The proof uses displays. For a commutative ring $R$ with unit, let $W(R)$ be its
$p$-typical Witt ring, and let $I_R$ be the kernel of the projection
$W(R)\to R$. The Frobenius endomorphism of $W(R)$ is
denoted by $\sigma$.

Let us first recall how the display of a $p$-divisible group over
$k$ is related to the covariant Dieudonn\'e modules of its
truncations. For a positive integer $n$, let $W_n=W_n(k)=W(k)/(p^n)$ be the
truncated Witt ring, and let $W_\infty=W(k)$. A Dieudonn\'e module
over $k$ of level $n\in\NN^{\ast}\cup\{\infty\}$ is a triple $(P,F,V)$
where $P$ is a free $W_n$-module of finite rank, $F:P\to P$ is
$\sigma$-linear,
$V:P\to P$ is $\sigma^{-1}$-linear, and we have $FV=VF=p$. If $n=1$, we require
that $\Ker(F)=\Image(V)$. A different way to represent these
objects is as follows.

\begin{Defn}
Let $n\in\NN^{\ast}\cup\{\infty\}$. A display over $k$ of level $n$ is a collection $(P,Q,\iota,\varepsilon,F_1)$, where $P$ and $Q$ are free
$W_n$-modules of the same finite rank, where $\iota:Q\to P$ and $\varepsilon:P\to Q$ are $W_n$-linear maps with $\iota\varepsilon=p$ and $\varepsilon\iota=p$,
and where $F_1:Q\to P$ is a $\sigma$-linear isomorphism. For
$n=1$, we require that $\Ker(\varepsilon)=\Image(\iota)$.
\end{Defn}

For $n\in\NN^{\ast}\cup\{\infty\}$, Dieudonn\'e modules over $k$ of level $n$
are equivalent to displays over $k$ of level $n$ via the association
$(P,Q,\iota,\varepsilon,F_1)\mapsto(P,F,V)$ with
$F=F_1\varepsilon$ and $V=\iota F_1^{-1}$. On the other hand,
Dieudonn\'e modules over $k$ of level $n$ are equivalent to $BT$ groups of
level $n$ over $k$. These equivalences are compatible with the
natural truncation operations on all sides. Displays over $k$ of level
$\infty$ are equivalent to (not necessarily nilpotent) displays
over $k$ in the sense of \cite{Zi1} because for $n=\infty$ we can identify $Q$ with the $W_n$-submodule $\iota(Q)$ of $P$.

The essence of the proof of Lemma \ref{Le-aux-fam} is the
following. Assume that for $n\in\NN^{\ast}$ we want to lift a display
$(P,Q,\iota,\varepsilon,F_1)$ over $k$ of level $n$ to a display
$(P',Q',\iota',\varepsilon',F_1')$ over $k$ of level $n+1$. 
It is easy to
see that the quadruple $(P',Q',\iota',\varepsilon')$ is uniquely
determined up to isomorphism and thus we can fix it. The set of
lifts $F_1^\prime$ of $F_1$ is a principally homogeneous space
under the $k$-vector space of $\sigma$-linear maps $Q\to p^nP'\cong
P/pP$.

\begin{proof}[Proof of Lemma \ref{Le-aux-fam}]
Let $X=\Spec R$. We recall that the category of infinitesimal $p$-divisible groups over $R$ is equivalent to the category of nilpotent displays over $R$, cf.\ \cite{Zi1} and \cite{La}. Let $\scrP=(P,Q,F,F_1)$ be the nilpotent display over $R$ associated to $\scrD$ i.e.,\ $\scrD=BT(\scrP)$ in the sense of \cite[Thm.~81]{Zi1}. 
Here $F:P\to P$ and $F_1:Q\to P$ are $\sigma$-linear maps of $W(R)$-modules. 
We choose a normal decomposition $P=J\oplus T$ 
of finitely generated projective $W(R)$-modules such that $Q=J\oplus
I_RT$. Let $\Psi:P\to P$ be the $\sigma$-linear automorphism
given by $F_1$ on $J$ and by $F$ on $T$, cf. \cite[Lemma 9]{Zi1}.

Let $E=\Hom_{W(R)}(P^{(\sigma)},P)$ where 
$P^{(\sigma)}=W(R)\otimes_{\sigma,W(R)}P$. 
We define $Y=\Spec R'$ to be the vector bundle over $X$
associated to the projective $R$-module $\bar E=E/I_RE$,
in other words $R'=\Sym^*(\bar E^\vee)$ with $\bar E^\vee=\Hom_R(\bar E,R)$.
For a $W(R)$-module $M$ we write $M_{R'}=W(R')\otimes_{W(R)}M$.

Let $u=\id\in\bar E\otimes_R\bar E^\vee\subset\bar E\otimes_R{R'}$ be the 
tautological section. Let $\tilde u\in E_{R'}$ be a lift of $u$
that maps to zero under the
$W(R)$-linear map $E_{R'}\to E$ induced naturally by the zero section of $Y$.
Let $\tilde\scrP=(P_{R'},J_{R'}\oplus I_{R'}T_{R'},\tilde F,\tilde F_1)$
be the display over ${R'}$ such that the $\sigma$-linear automorphism
$\tilde\Psi:P_{R'}\to P_{R'}$ defined by $\tilde F_1$ on $J_{R'}$ and by 
$\tilde F$ on $T_{R'}$ is equal to $\sigma\otimes\Psi+p^m\tilde u$; 
here we view $\tilde u\in E_{R'}$
as a $\sigma$-linear endomorphism of $P_{R'}$. 
Then $\tilde\scrP$ is again
a nilpotent display because the nilpotency condition depends
only on $\tilde\Psi$ modulo $p$.
Finally let $\scrE=BT(\tilde\scrP)$ be the associated
infinitesimal $p$-divisible group over $Y$.

We have to verify that properties (i) to (iii) hold.
Condition (i) follows from the construction.
The discussion preceding this proof implies
that (iii) and the following weakening of (ii) hold
(they suffice for our application).

\smallskip
(ii)$'$
If $y_1,y_2:\Spec \kappa\to Y$ are geometric points with
equal image in $X(\kappa)$, then there exists an isomorphism
$y_1^*\scrE[p^m]\cong y_2^*\scrE[p^m]$.

\smallskip
Condition (ii) follows from the following general lemma.
\end{proof}

\begin{Lemma}
\label{Le-disp-trunc}
Let $R$ be a ring in which $p$ is nilpotent
and let
$$
\scrP=(P,Q,F,F_1)\qquad\text{and}\qquad\scrP'=(P,Q,F',F_1')
$$
be two displays over $R$ with the same $W(R)$-modules. For a given normal
decomposition $P=J\oplus T$ i.e., $Q=J\oplus I_RT$, let $\Psi:P\to
P$ (resp.\ $\Psi':P\to P$) be the associated $\sigma$-linear
automorphism given by $F_1$ (resp.\ $F_1'$) on $J$ and by $F$
(resp.\ $F'$) on $T$. If\/ $\Psi'-\Psi=p^n\Omega$ for a $\sigma$-linear
endomorphism $\Omega:P\to P$, then there exists an isomorphism of fppf
sheaves over $R$ 
$$
BT(\scrP)[p^n]\cong BT(\scrP')[p^n].
$$
\end{Lemma}

\noindent Here $BT(\scrP)$ is the formal group over $R$ associated to $\scrP$ by
\cite[Thm.~81]{Zi1}, which is a $p$-divisible group when
$\scrP$ is nilpotent.

\begin{proof}
Let $P_2=P\oplus P$ and $Q_2=Q\oplus Q$ be endowed with the induced
normal decompositions $P_2=J\oplus T\oplus J\oplus T$
and $Q_2=J\oplus I_RT\oplus J\oplus I_RT$. We define two
displays $\scrK$ and $\scrK'$ with the same underlying
$W(R)$-modules $Q_2\subseteq P_2$ such that the associated
operators $\Psi$ with respect to these normal
decomposition are
$$
\Psi_{\!\scrK}=\left(\begin{matrix}
\Psi & -\Omega \\ 0 & \Psi' \end{matrix}\right)
\qquad\text{and}\qquad
\Psi_{\!\scrK'}=\left(\begin{matrix}
\Psi & 0 \\ \Omega & \Psi' \end{matrix}\right)
$$
as block matrices for $P_2=P\oplus P$. The matrix
$\Gamma=\left(\begin{smallmatrix}
0 & -1 \\ 1 & p^n \end{smallmatrix}\right)$
defines an isomorphism $\Gamma:\scrK'\cong\scrK$.
Indeed, as $\Gamma$ respects the normal
decompositions, this is equivalent to the
relation $\Gamma\Psi_{\!\scrK'}=\Psi_{\!\scrK}\Gamma$,
which is easily checked. Moreover, there exist
homomorphisms of displays
$$
\Theta:\scrK\to\scrP\oplus\scrP'
\qquad\text{and}\qquad
\Theta':\scrK'\to\scrP\oplus\scrP'
$$
given by the matrices
$\Theta=\left(\begin{smallmatrix}
p^n & 1 \\ 0 & 1 \end{smallmatrix}\right)$
and $\Theta'=\left(\begin{smallmatrix}
1 & 0 \\ 1 & p^n \end{smallmatrix}\right)$;
the required relations
$\Theta\Psi_{\!\scrK}=(\Psi\oplus\Psi')\Theta$
and $\Theta'\Psi_{\!\scrK'}=(\Psi\oplus\Psi')\Theta'$
are easily verified. We have $\Theta\Gamma=\Theta'$.
Consider the following homomorphisms of
complexes of displays:
$$
[\scrP'\xrightarrow{p^n}\scrP']
\xrightarrow{u'}
[\scrK'\xrightarrow{\Theta'}\scrP\oplus\scrP']
\cong
[\scrK\xrightarrow{\Theta}\scrP\oplus\scrP']
\xleftarrow{u}
[\scrP\xrightarrow{p^n}\scrP].
$$
The middle isomorphism is given by $\Gamma:\scrK'\cong\scrK$
and the identity of $\scrP\oplus\scrP'$.
The homomorphism $u$ is $x\mapsto(x,0)$ in both degrees,
while $u'$ is $x\mapsto(0,x)$ in both degrees.
We have an exact sequence of complexes of displays
$$
0\to[\scrP\xrightarrow{p^n}\scrP]\xrightarrow{u}
[\scrK\xrightarrow{\Theta}\scrP\oplus\scrP']\to
[\scrP'\xrightarrow{1}\scrP']\to 0
$$
and a similar one with the roles of $u$ and $\Theta$
taken by $u'$ and $\Theta'$. As the functor $BT$
preserves exact sequences, it follows that $BT(u)$ and $BT(u')$
are quasi-isomorphisms of complexes of fppf sheaves over $R$. 
Thus the complex $[p^n:BT(\scrP)\to BT(\scrP)]$ is quasi-isomorphic to 
$[p^n:BT(\scrP')\to BT(\scrP')]$, and the lemma holds.
\end{proof}

\subsection{Isogeny cutoffs}

\begin{Thm}
\label{Th-going-down-b}
Let $R=k[[t]]$
and let $\bar K$ be an algebraic closure of $K=k((t))$.
For every $p$-divisible group $\scrD$ over
$R$ with constant Newton polygon $\nu$ we have
$$
b_{\scrD_k}\leq b_{\scrD_{\bar K}}.
$$
\end{Thm}

\begin{proof}
We can assume that $\scrD$
is infinitesimal because there exists an exact sequence of
$p$-divisible groups
$0\to\scrD^\circ\to\scrD\to\scrD^{\text{\rm \'et}}\to 0$ with
$\scrD^{\text{\rm \'et}}$ \'etale and $\scrD^\circ_k$ connected. Then
$\scrD^\circ$ is infinitesimal because its Newton polygon is
constant. By Lemma \ref{Le-prelim}, we can replace $\scrD$ by
$\scrD^\circ$. 

We show that the assumption $b_{\scrD_k}>b_{\scrD_{\bar K}}$ leads
to a contradiction. We can assume that $\scrD$ has the
following maximality property: For each $p$-divisible group
$\scrD'$ over $R$ of constant Newton polygon $\nu$ with
$b_{\scrD_{\bar K}}=b_{\scrD'_{\bar K}}$, we have $b_{\scrD_k}\geq
b_{\scrD'_k}$. Let $m=b_{\scrD_k}-1>0$. We consider a vector bundle $\pi:Y\to X$ and an infinitesimal $p$-divisible group
$\scrE$ over $Y$ such that properties (i) to (iii) of Lemma \ref{Le-aux-fam} hold. Note that
$Y\cong\Spec R[t_1,\ldots,t_{h^2}]$, as $R$ is local.

We claim that $\scrE$ has constant Newton polygon $\nu$. Let
$Y_{\nu}\subseteq Y$ be the $\nu$-stratum of the Newton
polygon stratification on $Y$ defined by $\scrE$. As
$\scrE[p^m]\cong\pi^*\scrD[p^m]$ and as $m\geq b_{\scrD_{\bar
K}}$ by assumption, $Y_\nu$ contains all the $\bar K$-valued points of the generic fibre $Y_K$ and thus it contains $Y_K$. In particular, $Y_\nu$ is the unique open stratum.
Moreover, $Y_\nu$ contains the image of the zero section
$o:X\to Y$ as $o^*\scrE\cong\scrD$. Hence the complement of
$Y_\nu$ in $Y$ has codimension $\geq 2$. By the (weak)
purity theorem for Newton polygon strata (see \cite[Thm. 4.1]{dJO},
\cite[Thm.~1.6]{Va1}, or \cite{Zi2}) we get that $Y=Y_\nu$ as claimed.

Next we consider the function $b_\scrE:Y\to\NN$. As
$\scrE[p^m]\cong\pi^*\scrD[p^m]$, the relation $b_{\scrD_{\bar
K}}\leq m$ implies that $b_\scrE(y)=b_{\scrD_{\bar K}}$ for all
 closed points $y\in Y_K$, and the relation $b_{\scrD_k}>m$ implies that
$b_\scrE(y)>m$ for all closed points $y\in Y_k$. By the maximality property of
$\scrD$ it follows that $b_\scrE(y)=m+1$ for all closed points
$y\in Y_k$. Thus, as $\scrE$ has constant Newton polygon, from the property (iii) of Lemma \ref{Le-aux-fam} we get that $b_{\scrD_k}\leq m$. Contradiction.
\end{proof}

\subsection{Isomorphism numbers}

\begin{Thm}
\label{Th-going-down-n}
Let $R$, $K$,
and $\bar K$ be as in the previous theorem. 
For every $p$-divisible group $\scrD$ over
$R$ with constant Newton polygon $\nu$ we have
$$
n_{\scrD_k}\leq n_{\scrD_{\bar K}}.
$$
\end{Thm}

\begin{proof}
As in the proof of Theorem \ref{Th-going-down-b}
we can assume that $\scrD$ is infinitesimal.
Let $m=n_{\scrD_{\bar K}}$. Let
$E$ be a $p$-divisible group over $k$ such that
$E[p^m]$ is isomorphic to $\scrD_k[p^m]$.
By \cite[Thm.~4.4 f)]{Ill} there exists a $p$-divisible
group $\scrE$ over $R$ such that $\scrE_k\cong E$
and $\scrE[p^m]\cong\scrD[p^m]$.
By the choice of $m$, the last isomorphism implies that
$\scrE_{\bar K}$ and $\scrD_{\bar K}$ are isomorphic. Thus $q_{\scrD_{\bar K},\scrE_{\bar K}}=0$.
By Theorem \ref{Th-going-down-b} and
Lemma \ref{Le-prelim} (a) we have
$b_{\scrD_k}\leq b_{\scrD_{\bar K}}\leq m$.
Therefore $\scrE_k$ and $\scrD_k$ have the same Newton polygon,
which implies that $\scrE$ has constant Newton polygon $\nu$.
As $q_{\scrD_{\bar K},\scrE_{\bar K}}=0$, from Theorem \ref{Th-going-down-q} we get that $q_{\scrD_k,E}=0$ i.e., $E$ is
isomorphic to $\scrD_k$. Thus $n_{\scrD_k}\leq m$.
\end{proof}


\section{Semicontinuity}

In this section, we prove Theorem \ref{Th-semicont} and truncated variants of it.

\subsection{Lifting the level of truncated $BT$ groups}

Assume that $n\geq m\geq 1$ are integers.
Let $\scrX_n$ be the algebraic stack of $BT$
groups of level $n$ and let
$\tau:\scrX_n\to\scrX_m$ be the truncation
morphism.

\begin{Defn}
Let $\scrB$ be a $BT$ group of level $m$ over a scheme $X$.
An \emph{exhaustive extension} of $\scrB$ to level $n$ is a
$BT$ group $\scrC$ of level $n$ over an $X$-scheme $Y$
together with an isomorphism $\scrB_Y\cong\scrC[p^m]$
such that the induced morphism
$Y\to X\times_{\scrX_m}\scrX_n$
is surjective on geometric points.
\end{Defn}

\noindent
It is easy to see that each $\scrB$ over $X$ as above has an
exhaustive extension to level $n$ over an affine $X$-scheme $Y$ of finite
type. Namely, if $Z\to\scrX_n$ is an affine smooth presentation, we can
take $Y=X\times_{\scrX_m}Z$. As $\tau$ is smooth and surjective by
\cite[Thm.~4.4]{Ill}, $Y\to X$ is as well affine smooth and surjective.  

With more effort, one can arrange that the universal $BT$ group of level
$n$ over $Z$ (and thus also the exhaustive extension $\scrC$ over
$Y$) comes from a $p$-divisible group (cf. \cite[Prop.~2.3]{NVW} via a natural passage to an affine open cover).
Here we need only the following consequence of this fact, which
can be deduced from the previous paragraph by a limit argument.

\begin{Lemma}
\label{Le-extend}
For a $BT$ group $\scrB$ of level $m$
over a scheme $X$ there exists a faithfully flat
affine morphism $Y\to X$ and a $p$-divisible group
$\scrD$ over $Y$ such that $\scrB_Y\cong\scrD[p^m]$.
\end{Lemma}

Next we show that having a unique lift
(up to isomorphism or up to isogeny)
is a constructible property of truncated
$BT$ groups.

\begin{Prop}
\label{Pr-constr-isom}
Let $\scrB$ be a $BT$ group of level $m$ over a scheme $X$
and let $n\geq m$ be an integer. There exists a constructible
subset $U$ of $X$ such that a geometric point
$\bar x:\Spec\kappa\to X$ lies in $U(\kappa)$ if and only if
all $BT$ groups of level $n$ over $\kappa$ which extend the 
geometric fibre $\scrB_{\bar x}$ are isomorphic.
\end{Prop}

This includes the following invariance under field extensions.

\begin{Cor}
\label{Co-constr-isom}
Let $k\subseteq\kappa$ be an extension of algebraically closed
fields. A $BT$ group $B$ of level $m$ over $k$ extends uniquely to
level $n$ (up to isomorphism) if and only if $B_\kappa$ extends 
uniquely to level $n$ (up to isomorphism). \qed
\end{Cor}

\begin{proof}[Proof of Proposition \ref{Pr-constr-isom}]
We can assume that $X$ is of finite type over $\Spec\ZZ$.
Let $Y\to X$ be a morphism of finite type such that over $Y$
there exists an exhaustive extension $\scrC$ of $\scrB$ to level $n$.
Let $Y'=Y\times_XY$, let $p_1,p_2:Y'\to Y$
be the two projections, and consider
$$
Z=\uIsom(p_1^*\scrC,p_2^*\scrC)\xrightarrow{\psi} Y'.
$$
Let $x\in X$ be the image of $\bar x:\Spec\kappa\to X$.
The geometric fibre $\scrB_{\bar x}$ extends uniquely to 
level $n$ (up to isomorphism) if and only if the geometric
fibre $\psi_{\bar x}:Z_{\bar x}\to Y'_{\bar x}$ is surjective
on $\kappa$-valued points,
which is equivalent to the fibre $\psi_x:Z_x\to Y'_x$
being surjective. Hence $U$ is a well-defined subset of $X$,
and $X\setminus U$ is the image of $Y'\setminus\Image(\psi)\to X$.
As $X,Y',Z$ are of finite type, $U$ is constructible by 
Chevalley's theorem.
\end{proof}

\begin{Defn}
\label{Def-bn-truncated}
We say that a $BT$ group $B$ of level $m$ over $k$ has well-defined 
Newton polygon $\nu$ if each $p$-divisible
group $D$ over $k$ with $D[p^m]$ isomorphic to $B$ has Newton
polygon $\nu$. In this case we let $b_B=b_D$. Similarly, if all
$p$-divisible groups $D$ with $D[p^m]\cong B$ are isomorphic, we
let $n_B=n_D$, while $n_B$ is undefined otherwise.
\end{Defn}

\begin{Prop}
\label{Pr-constr-isog}
Let $\scrB$ be a $BT$ group of level $m$ over an $\FF_p$-scheme $X$.
There exists a constructible subset $U$ of $X$ such that
a geometric point $\bar x:\Spec\kappa\to X$ lies in
$U(\kappa)$ if and only if the geometric fibre 
$\scrB_{\bar x}$ has a well-defined Newton polygon.
\end{Prop}

Proposition \ref {Pr-constr-isog} includes the following invariance under field extensions.

\begin{Cor}
\label{Co-constr-isog}
Let $k\subset\kappa$ be an extension of algebraically closed
fields. A $BT$ group $B$ of level $m$ over $k$ has a well-defined Newton
polygon if and only if $B_\kappa$ has a well-defined 
Newton polygon. \qed
\end{Cor}

The proof of Proposition \ref{Pr-constr-isog}
uses the following standard fact.

\begin{Lemma}
\label{Le-Newton-trunc}
Let $\scrB$ be a $BT$ group of level $n$ over an $\FF_p$-scheme $X$
such that for each geometric point $\bar x:\Spec\kappa\to X$, 
the geometric fibre $\scrB_{\bar x}$
has a well-defined Newton polygon $\nu(\bar x)$.
If $x\in X$ is the image of $\bar x$, then $\nu(x):=\nu(\bar x)$ is well-defined. Moreover, for each Newton polygon $\nu$ the set
$X_\nu=\{x\in X\mid\nu(x)\preceq\nu\}$ is closed in $X$.
\end{Lemma}

We recall that $\nu'\preceq\nu$ if and only if the
polygons $\nu'$ and $\nu$ share the same endpoints
and all points of $\nu'$ lie on or above $\nu$.

\begin{proof}
As Newton polygons of $p$-divisible groups are preserved under extensions of the base field, $\nu(x)$ is well-defined.
By Lemma \ref{Le-extend} there exists a faithfully flat
affine morphism $f:Y\to X$ such that $\scrB_Y$ extends
to a $p$-divisible group $\scrD$ over $Y$.
Then $\pi^{-1}(X_\nu)$ is the set of points
where the Newton polygon of $\scrD$ is $\preceq\nu$,
which is closed in $Y$ by \cite[Thm.~2.3.1]{Ka} applied to
the Dieudonn\'e $F$-crystal of $\Dscr$.
It follows that $X_\nu$ is closed in $X$ as $f$ is faithfully flat.
\end{proof}

\begin{proof}[Proof of Proposition \ref{Pr-constr-isog}]
We can assume that $X$ is of finite type over $\Spec\FF_p$.
Choose $n\in\NN^{\ast}$ such that for each geometric point
$\bar x:\Spec\kappa\to X$ and every $p$-divisible group $D$
over $\kappa$ that extends $\scrB_{\bar x}$, we have $b_D\leq n$.
This is possible because $\scrB$ has bounded height. 
Let $Y\to X$ be a morphism of finite type such that over $Y$
there exists an exhaustive extension $\scrC$ of $\scrB$ to level $n$.
For $y\in Y$ let $\nu(y)$ be the well-defined Newton polygon
of the fibre $\scrC_y$. By Lemma \ref{Le-Newton-trunc} the set
$Y^\nu=\{y\in Y\mid \nu(y)\neq \nu\}$ is locally closed in $Y$. 
Let $U_\nu\subseteq X$ be the complement of the image of
$Y^\nu\to X$. This is a constructible set by Chevalley's
theorem. The required subset $U$ of $X$ is the union of all $U_\nu$,
which is constructible because only finitely many $U_\nu$'s are non-empty.
\end{proof}

\begin{Cor}
\label{Co-perm}
Let $k\subseteq\kappa$ be an extension of algebraically
closed fields. For a $p$-divisible group $D$ over $k$
we have $n_D=n_{D_\kappa}$ and $b_D=b_{D_\kappa}$.
\end{Cor}

\begin{proof}
Corollary \ref{Co-constr-isom} applied to $n=\max\{n_D,n_{D_\kappa}\}$ and $B=D[p^m]$ 
with $m=\min\{n_D,n_{D_\kappa}\}$ gives $n_D=n_{D_\kappa}$. 
Corollary \ref{Co-constr-isog} applied to $B=D[p^m]$ with
$m=\min\{b_D,b_{D_\kappa}\}$ gives $b_D=b_{D_\kappa}$.
\end{proof}

\subsection{Distances of truncated $BT$ groups}

In order to deduce semicontinuity of distance numbers from their going down property 
we need to define distance numbers also for truncated $BT$ groups.

\begin{Defn}
Let $B$ and $C$ be truncated $BT$ groups of level $n$ over $k$
The distance $q_{B,C}$ is the smallest non-negative
integer $m$ such that there exist homomorphisms
$B\to C$ and $C\to B$ whose kernels are annihilated by $p^m$.
\end{Defn}

If $\kappa$ is an algebraically closed extension of $k$, we have $q_{B,C}=q_{B_{\kappa},C_{\kappa}}$. 
This can be viewed as a consequence of the following proposition.

\begin{Prop}
\label{Pr-constr-q-trunc}
Let $\scrB$ and $\scrC$ be truncated $BT$ groups of
level $n$ over an $\FF_p$-scheme $X$.
There exists a constructible subset $U$ of $X$ such that
a geometric point $\bar x:\Spec\kappa\to X$ lies in $U(\kappa)$
if and only if $q_{\scrB_{\bar x},\scrC_{\bar x}}\leq m$.
\end{Prop}

\begin{proof}
We can assume that $X$ is of finite type over $\Spec\FF_p$.
Let
$$
Z\subseteq\uHom(\scrB,\scrC)
\times\uHom(\scrC,\scrB)
$$
be the subscheme of all pairs of homomorphisms $(f,g)$ such
that $p^m$ annihilates both $\Ker(f)$ and $\Ker(g)$.
This is a closed subscheme.
The set $U$ is the image of $Z\to X$, which
is constructible by Chevalley's theorem.
\end{proof}

\begin{Prop}
\label{Pr-q-trunc}
Let $h\in\NN$.
Let $D$ and $E$ be $p$-divisible groups over $k$ of height at most $h$.
Let $B=D[p^n]$ and $C=E[p^n]$ for some $n\in\NN^{\ast}$.

\smallskip
(a) {} We have $q_{B,C}\leq q_{D,E}$.

\smallskip
(b) {} If $q_{B,C}<n-N_h$, then $q_{B,C}=q_{D,E}$.
\end{Prop}

\noindent
The number $N_h$ was defined in Section \ref{Se-prelim}.

\begin{proof}
(a) Let $m=q_{D,E}$ and let $\rho:D\to E$ be an isogeny
such that $\Ker(\rho)$ is contained in $D[p^m]$.
The isogenies $\rho$
and $p^m\rho^{-1}$ induce homomorphisms $B\to C$ and
$C\to B$ with kernels annihilated by $p^m$; thus
$q_{B,C}\leq m$.

(b) Let $m=q_{B,C}$ and let $f:B\to C$ and
$g:C\to B$ be homomorphisms with kernels annihilated
by $p^m$. We must show that $q_{D,E}\leq m$.
By the choice of $N_h$, there exist homomorphisms
$f':D\to E$ and $g':E\to D$ which coincide with
$f$ and $g$ on $D[p^{n-N_h}]$ and $E[p^{n-N_h}]$ (respectively).
As $m<n-N_h$, by Lemma \ref{Le-q-aux} below it follows that 
$\Ker(f')=\Ker(f)$ and $\Ker(g')=\Ker(g)$ are annihilated by $p^m$.
Hence $q_{D,E}\leq m$ as required.
\end{proof}

\begin{Lemma}
\label{Le-q-aux}
For a positive integer $l$, let $B$ and $C$ be truncated
$BT$ groups over $k$ of level $l+1$. Let $f:B\to C$ be a
homomorphism with restriction $f_l:B[p^l]\to C[p^l]$.
If $\Ker(f_l)\subseteq B[p^{l-1}]$, then we have $\Ker(f)=\Ker(f_l)$.
\end{Lemma}

\begin{proof}
Let $S$ be a $k$-scheme.
If $x\in\Ker(f)(S)$, then $px\in\Ker(f_l)(S)$, thus $px\in B[p^{l-1}](S)$
by the assumption. Hence $x\in B[p^l](S)$, which implies
that $x\in\Ker(f_l)(S)$ as $x\in\Ker(f)(S)$.
\end{proof}

\subsection{Proof of the semicontinuity results}
\label{Se-pf-semicont}

If $\Bscr$ is a truncated $BT$ group over an $\FF_p$-scheme $S$,
for each point $s\in S$ we write $b_\Bscr(s)=b_{\Bscr_{\bar s}}$ and
$n_\Bscr(s)=n_{\Bscr_{\bar s}}$, where $\Bscr_{\bar s}$ is the geometric
fibre of $\Bscr$ at $s$; see Definition \ref{Def-bn-truncated}.

\begin{Thm}
\label{Th-semicont-trunc} Let $m\leq l$ be positive integers. Let
$\scrB$ be a $BT$ group of level $l$ over
an $\FF_p$-scheme $S$ with well-defined and constant Newton
polygon $\nu$ i.e., all geometric fibres of $\scrB$ have
well-defined Newton polygon $\nu$. Then the following two properties hold:

\smallskip
(a) The set $U_{b_\scrB}= \{s\in S\mid b_\scrB(s)\leq m\}$ is
closed in $S$.

\smallskip
(b) The set $U_{n_\scrB}= \{s\in S\mid \text{$n_\scrB(s)$ is
defined and $\,\leq m$}\}$ is closed in $S$.
\end{Thm}

\begin{proof}
Let $\Box\in\{n,b\}$. The set $U_{\Box_\scrB}$ is functorial
in the sense that for a morphism $\pi:S'\to S$ we have
$\pi^{-1}U_{\Box_\scrB}=U_{\Box_{\pi^*(\scrB)}}$; here
we use Corollary \ref{Co-perm}. Thus we can assume that $S$ 
is of finite type over $\Spec\FF_p$. 
The set $U_{b_\scrB}$ (resp.\ $U_{n_\scrB}$) is constructible because
it coincides with the set $U$ associated to $\scrB[p^m]$ in
Proposition \ref{Pr-constr-isog} (resp.\ in Proposition
\ref{Pr-constr-isom} applied with sufficiently large $n$). Hence
it suffices to show that $U_{\Box_\scrB}$ is
stable under specialization. To prove this we can assume that
$S=\Spec k[[t]]$ for some algebraically closed field $k$; here
we use functoriality again. 
By \cite[Thm.~4.4]{Ill} there exists a $p$-divisible group $\scrD$
over $S$ that extends $\scrB$. As for $s\in S$ we have
$b_{\scrB_{\bar s}}=b_{\scrD_{\bar s}}$, and $n_{\scrB_{\bar s}}=n_{\scrD_{\bar s}}$ when $n_{\scrD_{\bar s}}\leq l$ while
$n_{\scrB_{\bar s}}$ is undefined otherwise, Theorem \ref{Th-semicont-trunc} follows
from Theorems \ref{Th-going-down-b} and \ref{Th-going-down-n}.
\end{proof}

We note that the results of \cite{GV} and \cite{Va2}
also allow a short and different proof of Theorem \ref{Th-semicont-trunc} (b) 
(see Remark \ref{7.13}). For distance numbers we have the following analogue of
Theorem \ref{Th-semicont-trunc}.

\begin{Thm}
\label{Th-semicont-trunc-q}
Let $\scrB$ and $\scrC$ be truncated $BT$ groups of level $n$
over an $\FF_p$-scheme $S$ with well-defined and constant
Newton polygons and height $\leq h$. 
Then for each non-negative integer $m<n-N_h$
the set
$$
U_{q_{\scrB,\scrC}} =\{s\in S\mid q_{\scrB,\scrC}(s)\leq m\}
$$
is closed in $S$. Here we write 
$q_{\scrB,\scrC}(s)=q_{\scrB_{\bar s},\scrC_{\bar s}}$ as above.
\end{Thm}

\begin{proof}
We can assume that $S$ is of finite type over $\Spec\FF_p$.
The set
$U_{q_{\scrB,\scrC}}$ is constructible, cf.\ Proposition \ref{Pr-constr-q-trunc}. Thus, as in the previous
proof, we can assume that $S=\Spec k[[t]]$ and that
$\scrB=\scrD[p^n]$ and $\scrC=\scrE[p^n]$ for $p$-divisible groups
$\scrD$ and $\scrE$ over $S$. By Proposition \ref{Pr-q-trunc}, for
$s\in S$ we have $s\in U_{q_{\scrB,\scrC}}$ if and only if
$q_{\scrD_{\bar s},\scrE_{\bar s}}\leq m$. Thus
$U_{q_{\scrB,\scrC}}$ is stable under specialization by Theorem
\ref{Th-going-down-q}. We conclude that $U_{q_{\scrB,\scrC}}$ is closed in $S$.
\end{proof}

\begin{proof}[Proof of Theorem \ref{Th-semicont}]
We can assume that $S$ is quasi-compact. Then $\scrD$ and $\scrE$
have height $\leq h$ for some integer $h$. We choose an
integer $l\geq m$ such that $\scrD[p^l]=\scrB$ and $\scrE[p^l]=\scrC$ have
well-defined (necessarily constant) Newton polygons. Then the sets
$U_{b_\scrD}$ and $U_{n_\scrD}$ coincide with the sets
$U_{b_\scrB}$ and $U_{n_\scrB}$ (respectively) considered in Theorem \ref{Th-semicont-trunc}, which implies
that $U_{b_\scrD}$ and $U_{n_\scrD}$ are closed. Clearly (c)
follows from (d). To prove (d) we assume in addition that
$l>m+N_h$. By Proposition \ref{Pr-q-trunc} it follows that
$U_{q_{\scrD, \scrE}}$ coincides with the set
$U_{q_{\scrB,\scrC}}$ in Theorem \ref{Th-semicont-trunc-q}, 
which implies that $U_{q_{\scrD, \scrE}}$ is closed.
\end{proof}


\section{Complements on Dieudonn\'e modules}
\label{Complem}

In this section we collect several properties of Dieudonn\'e
modules which will be used later on to study $n_D$ and $b_D$.
The results of Subsections \ref{Subsec-val} and \ref{Subsec-pres-cyc}
are either well-known or trivial. Subsections \ref{Subsec-val-cyc}
and \ref{Subsec-min-mod} are more involved and contain new material.

In this section, to be short we write $W=W(k)$ and 
$W_\QQ=W(k)[1/p]$.
Let $\sigma:W\to W$ be the Frobenius automorphism and
let $v:{W_\QQ}\to\ZZ\cup\{\infty\}$ be the $p$-adic valuation. 
Let ${W_\QQ}\{F,F^{-1}\}$ be the non-commutative Laurent
polynomial ring. We consider the quotient ring
$$
\DD={W_\QQ}\{F,F^{-1}\}/I
$$
where $I$ is the two-sided ideal generated by
all elements $Fa-\sigma(a)F$ with $a\in {W_\QQ}$.
Let $\EE\subset\DD$ be the $W$-subalgebra generated by $F$ and $V=pF^{-1}$.

\subsection{Valuations on Frobenius modules}
\label{Subsec-val}

A {\em valuation} on a $W$-module $M$ is a map
$w:M\to\RR\cup\{\infty\}$ that has the following two properties:

\smallskip
(i) $w(ax)=v(a)+w(x)$ for all $a\in W$ and $x\in M$;

\smallskip
(ii) $w(x+y)\geq\min\{w(x),w(y)\}$ for all $x,y\in M$.

\smallskip
\noindent The valuation is called non-degenerate if $w(x)=\infty$
implies that $x=0$. It is called non-trivial if $w(x)\neq\infty$
for some $x\in M$. If $M_{\rm tors}$ is the maximal torsion $W$-submodule of $M$, then $w$ always factors through $M/M_{\rm tors}$. 
Denoting $M_\QQ=M\otimes_W {W_\QQ}$, 
valuations on $M$ extend uniquely to valuations on $M_\QQ$. 
If $M$ is a $W_\QQ$-vector space, then (i) holds for all $a\in {W_\QQ}$.

\begin{Defn}
Let $w$ be a valuation on a $W$-module $M$.
A direct sum decomposition $M=\bigoplus_{i\in I} M_i$ is called
\emph{valuative} if we have
$$
w\Big(\sum_{i\in I} m_i\Big)=\min\{w(m_i)\mid i\in I\}$$ 
whenever $m_i\in M_i$ with almost all $m_i=0$.
A $W$-basis or ${W_\QQ}$-basis of $M$ is called valuative
if the associated direct sum decomposition into modules of
rank one is valuative.
\end{Defn}

A direct sum decomposition $M=\bigoplus_{i\in I} M_i$ is valuative if and only if
$w$ is minimal among all valuations of $M$
that coincide with $w$ on each $M_i$.
A direct sum decomposition of $M$ is valuative if and only if the
induced direct sum decomposition of $M_\QQ$ is valuative.

\begin{Defn}
A pair $(M,F)$, where $M$ is a $W$-module and $F$ is a
$\sigma$-linear endomorphism of $M$, is called a Frobenius
module (over $k$). A valuation $w$ on $M$ is called an {\em $F$-valuation
of slope $\lambda\in\RR$} if for all $x\in M$ we have 
$$w(Fx)=w(x)+\lambda.$$
\end{Defn}

An $F$-valuation on $M$ extends uniquely to
an $F$-valuation on $M_\QQ$.
We view each left $\DD$-module as a Frobenius module.
For each $\lambda\in\RR$ there exists a unique $F$-valuation
$\ww_\lambda$ on $\DD$ of slope $\lambda$ such that
the ${W_\QQ}$-basis $(F^i)_{i\in\ZZ}$ of $\DD$ is valuative and we have
$\ww_\lambda(1)=0$. It is given by the formula
$$
\ww_\lambda\Big(\sum_{i\in\ZZ} e_iF^i\Big)=\min\{v(e_i)+i\lambda\mid i\in\ZZ\}
$$
whenever $e_i\in W_\QQ$ with almost all $e_i=0$.
This can also be expressed in terms of Newton polygons.
For a non-zero element $\Phi\in\DD$ we define its Newton polygon
$\nu_\Phi$ as follows. Write
$\Phi=\sum_{i=n}^{m}a_iF^{-i}$ with $n\le m$
and $a_i\in W_\QQ$
such that $a_n\ne 0$ and $a_m\ne 0$.
Then $\nu_\Phi:[n,m]\to\RR$ is
the maximal upper convex function such that 
$v(a_i)\geq \nu_{\Phi}(i)$ for each integer
$i\in[n,m]$. 
For $\lambda\in\RR$ let $\nu_{\Phi,\lambda}:\RR\to\RR$
be the maximal linear function of slope $\lambda$
such that $v(a_i)\geq\nu_{\Phi,\lambda}(i)$
for each integer $i\in[n,m]$. For $t\in[n,m]$ we have 
$$
\nu_{\Phi}(t)=\max\{\nu_{\Phi,\lambda}(t)\mid\lambda\in\RR\},
$$
and the functions $\nu_{\Phi,\lambda}$
are maximal with this property. We have
\begin{equation}
\label{Eq-ww-lambda}
\ww_{\lambda}(\Phi)=\nu_{\Phi,\lambda}(0).
\end{equation}

In the following let $N$ be a left $\DD$-module of
finite dimension over ${W_\QQ}$.

\begin{Lemma}
\label{Le-F-val}
There exists a non-degenerate $F$-valuation of slope $\lambda$
on $N$ if and only if $N$ is isoclinic of slope $\lambda$.
When $N$ is simple of slope $\lambda$, then any two non-trivial
$F$-valuations on $N$ differ by the addition of a constant.
\end{Lemma}

\begin{proof}[Sketch of proof]
Use the facts that $N$ has a ${W_\QQ}$-basis
consisting of elements $x$ with $F^{s_x}(x)=p^{r_x}x$
for some integers $s_x\ne 0$ and $r_x$, and that
$N$ is isoclinic of slope $\lambda$ if and only if
we have $r_x=\lambda s_x$ for all $x$ in the ${W_\QQ}$-basis.
\end{proof}

\begin{Lemma}
\label{Le-min-val} Let $N$ be as above and let $\lambda\in\RR$. We
consider a free $W$-submodule $M\subset N$ with $M_\QQ=N$. Let
$\scrW$ be the set of all $F$-valuations $w$ of slope $\lambda$ on
$N$ with $w(x)\geq 0$ for all $x\in M$. Then $\scrW$ has a minimal
element $w_\circ$ i.e., we have $w_\circ(x)\leq w(x)$ for all $w\in\scrW$
and $x\in N$. The valuation $w_\circ$ is non-degenerate if and
only if $N$ is isoclinic of slope $\lambda$.
\end{Lemma}

\begin{proof}
For $x\in N$ let $w_\circ(x)=\inf\{w(x)\mid w\in\scrW\}$.
Then $w_\circ$ is an $F$-valuation on $N$ of slope $\lambda$.
The last assertion follows from Lemma \ref{Le-F-val}.
\end{proof}

Let $N$ be as above. For an $F$-valuation $w$ on $N$
of slope $\lambda$ we consider the $W$-submodules of $N$
defined for each $\alpha\in\RR$ by:
\begin{gather*}
N^{w\geq\alpha}=\{x\in N\mid w(x)\geq\alpha\}\qquad\text{and}\qquad
N^{w>\alpha}=\{x\in N\mid w(x)>\alpha\}.
\end{gather*}
Let $gr^\alpha N$ be the $k$-vector
space $N^{w\geq\alpha}/N^{w>\alpha}$.

\begin{Lemma}
\label{Le-Phi}
Let $\Phi\in\DD\setminus\{0\}$ be a sum $\Phi=\sum_{i\in\ZZ}  e_iF^i$
with each $e_i\in {W_\QQ}$ (only a finite number of the $e_i$'s are non-zero). 
Let $\delta=\ww_\lambda(\Phi)$.
Then the multiplication by $\Phi$ induces a group homomorphism
$$
\bar\Phi_\alpha:gr^\alpha N\to gr^{\alpha+\delta}N
$$
which is surjective with finite kernel.
\end{Lemma}

\begin{proof}
Clearly $\Phi(N^{w\geq\alpha})\subseteq N^{w\geq\alpha+\delta}$
and $\Phi(N^{w>\alpha})\subseteq N^{w>\alpha+\delta}$.
The induced homomorphism $\bar\Phi_\alpha$ does not change if we
omit from $\Phi$ all terms $e_iF^i$ with $v(e_i)+i\lambda>\delta$. As the
validity of the lemma is invariant under multiplying $\Phi$ with
integral powers of $F$ or with non-zero elements of ${W_\QQ}$, we can assume that
$\delta=0$ and that $\Phi=1+\sum_{i\in\NN^{\ast}} e_iF^i$ 
with $e_i\in W_\QQ$ such that only a finite number of them 
are non-zero and we have $v(e_i)=-i\lambda$ for all non-zero $e_i$. 
Then $\bar\Phi_\alpha$ can be viewed naturally as an \'etale endomorphism
of the vector group scheme over $k$ defined by
$gr^{\alpha}N$, and the assertion follows.
\end{proof}

For later use we record the following elementary result.

\begin{Lemma}
\label{Le-p-exp}
Let $\alpha, \alpha'\in \RR$ be such that
$gr^\alpha N\neq 0$. The minimal non-negative
integer $m$ such that $p^mN^{w\ge\alpha}\subseteq N^{w>\alpha'}$
is equal to $\lfloor\alpha'-\alpha\rfloor+1$. \qed
\end{Lemma}

\subsection{Presentations of cyclic Dieudonn\'e modules}
\label{Subsec-pres-cyc}

In this subsection we fix a non-zero bi-nilpotent Dieudonn\'e module $M$ over
$k$ i.e., $M$ is an $\EE$-module which is a free $W$-module
of finite positive rank, and $F$ and $V$ are nilpotent on $\bar M=M/pM$. Let
$d=\dim_k(M/FM)$ and $c=\dim_k(M/VM)$ and $h=c+d$. We have $cd>0$,
and $h$ is the rank of $M$. 

The $a$-number $a(M)=\dim_k(M/(FM+VM))$ is positive.
An element $z\in M$ generates $M$ as an $\EE$-module
if and only if $z$ generates $M/(FM+VM)$ as a $k$-vector space.
For completeness we prove the following
well-known lemma.

\begin{Lemma}
\label{Le-present}
Assume that $z$ generates $M$ as an $\EE$-module.
Then the following three properties hold:

\smallskip
(a) The $h$-tuple $\Upsilon=(F^iz)_{1\le i\le c}\sqcup(V^iz)_{0\leq i\leq d-1}$
is a $W$-basis of $M$.

\smallskip
(b) There exists an element $\Psi\in\EE$ for which we have
$\Psi z=0$ and which is of the form
$$
\Psi=\sum_{i=0}^ca_iF^{c-i}+\sum_{i=1}^{d}b_iV^i,
$$
with $a_0$ and $b_d$ as units in $W$
and with $a_i\in pW$ for $i\in\{1,\ldots,c\}$ and $b_i\in pW$ for $i\in \{1,\ldots,d-1\}$.

\smallskip
(c) We have an $\EE$-linear isomorphism
$\EE/\EE\Psi\cong M$ given by $1+\EE\Psi\mapsto z$.
\end{Lemma}

\begin{proof}
As $V$ is nilpotent on $M/FM$ and as
$M/FM$ is a $k$-vector space of dimension $d$
generated by the iterates of $z$ under $V$,
we get that $(V^iz)_{0\leq i<d}$ is a $k$-basis of $M/FM$.
Similarly we argue that $(F^iz)_{0\leq i<c}$ is a
$k$-basis of $M/VM$, which implies that
$(F^iz)_{1\leq i\leq c}$ is a $k$-basis of $FM/pM$.
We conclude that $\Upsilon$ is a $k$-basis of $M/pM$.
From this (a) follows.

By (a) there exists a relation $\Psi z=0$ with
$\Psi=\sum_{i=0}^c a_iF^{c-i}+\sum_{i=1}^d b_iV^i$ and $b_d=1$. As $V^dz\in FM$ we have
$a_c\in pW$ and $b_i\in pW$ for $i\in \{1,\ldots,d-1\}$.
By interchanging the roles of $F$ and $V$ in (a), there exists also a
relation $\Psi'z=0$ with $\Psi'=\sum_{i=0}^c a_i'F^{c-i}+\sum_{i=1}^d b_i'V^i$ and $a_0'=1$ and such that
$a_i'\in pW$ for $i\in\{1,\ldots,c\}$. The element $\Psi'-b_d'\Psi$ must be zero by (a),
which implies that $a_i'=b_d'a_i$ for all $i\in\{0,\ldots,c\}$. As $a_0'=1$, we get that $b_d'$ and $a_0$ are units of $W$ and that we have $a_i\in pW$ for all $i\in\{1,\ldots,c-1\}$. Therefore (b) holds.

We have an $\EE$-linear epimorphism $\EE/\EE\Psi\to M$ that maps $1+\EE\Psi$ to $z$.
It is easy to see that 
$\{F^i\mid1\leq i\leq c\}\cup\{V^i\mid0\leq i\leq d-1\}$ 
generates $\EE/\Psi\EE$ over $W$, which proves (c).
\end{proof}

\begin{Lemma}
\label{Le-Newton}
Let $M=\EE/\EE\Psi$ be as in Lemma \ref{Le-present}.
Let $\nu_M:[0,h]\to\RR$ be the Newton polygon of $M$
and let $\nu_{\Psi}:[-c,d]\to\RR$ be the Newton
polygon of $\Psi$ defined above. Then
$\nu_M(t)=\nu_\Psi(t-c)$ for $t\in[0,h]$.
\end{Lemma}

\begin{proof}
This is proved in \cite[Lemma~2 on p.~82]{De}.
\end{proof}

\begin{Remark}
In view of Lemma \ref{Le-Newton} it would be natural
to shift the Newton polygon of $M$ so that its domain
is $[-c,d]$. This would cause $c$ to be replaced by $0$ in
many formulas, including the assertions of
Theorems \ref{Th-b-maximal}, \ref{Th-n-upper}, and 
\ref{Th-min-isog}. 
We keep the traditional notation in order to avoid confusion.
\end{Remark}

\subsection{Valuations on cyclic Dieudonn\'e modules}
\label{Subsec-val-cyc}

We assume now that $M$ is a non-zero bi-nilpotent
Dieudonn\'e module over $k$ generated
as an $\EE$-module by a fixed element $z\in M$.
Thus $M\cong\EE/\Psi\EE$ with $\Psi$ as in Lemma \ref{Le-present} (b)
and we have $a(M)=1$.

\begin{Lemma}
\label{Le-val-basis}
Assume that $M$ is isoclinic of slope $\lambda$.
Let $w$ be the minimal $F$-valuation on $M$ of
slope $\lambda$ with $w(x)\geq 0$ for all $x\in M$,
cf.\ Lemma \ref{Le-min-val}. Then the $W$-basis $\Upsilon$
of $M$ introduced in Lemma \ref{Le-present} (a)
is valuative for $w$.
\end{Lemma}

\begin{proof}
We write $\Psi=\sum_{i=0}^ha_iF^{c-i}$ with $a_i\in W$; 
thus for $i\in\{1,\ldots,d\}$ we have $a_{c+i}=p^ib_i$.
For $x=\sum_{i=0}^{h-1} e_iF^{c-i}z\in M$ with $e_i\in W$
let
$$
w_1(x)=\min\{v(e_i)+(c-i)\lambda\mid 0\le i\le h-1\}.
$$
Then $w_1$ is a valuation for which the $W$-basis $\Upsilon$
is valuative.
We claim that $w=w_1$. It is easy to see that $w(x)\geq w_1(x)\geq
0$ for all $x\in M$. Hence we must show that $w_1$ is an
$F$-valuation of slope $\lambda$ i.e., that we have
$w_1(Fx)=w_1(x)+\lambda$. This is a straightforward computation based on the relation $v(a_i)\geq i\lambda$ for all $0\leq i\leq h$ with equality for $i=0$ and $i=h$. The details are left to the reader.
\end{proof}

For the general (non-isoclinic) case we need some additional notations.

\begin{Notation}
\label{Notat}
Let $N=M_\QQ$ and
let $N=N_1\oplus\cdots\oplus N_r$ be the direct sum decomposition
into isoclinic components, ordered such that each $N_j$ with $j\in\{1,\ldots,r\}$ is isoclinic of slope $\lambda_j$ and
$0<\lambda_1<\cdots<\lambda_r<1$.
Let $h_j$ be the dimension of $N_j$ i.e., the multiplicity of $\lambda_j$ in $\nu$.
We write $h_j=c_j+d_j$ such that $\lambda_j=d_j/h_j$.
Let $M_j\subseteq N_j$ be the image of $M$.
Then $a(M_j)=1$. Let $w_j$ be the minimal $F$-valuation
on $N_j$ of slope $\lambda_j$ such that $w_j(M_j)\geq 0$,
see Lemma \ref{Le-min-val}.
For $\alpha=(\alpha_1,\ldots,\alpha_r)\in\RR^r$ let
$$
N^\alpha=\bigoplus_{j=1}^rN_j^{w_j\geq\alpha_j},\qquad
N^{\alpha+}=\bigoplus_{j=1}^rN_j^{w_j>\alpha_j},
$$
and $gr^\alpha N=N^\alpha/N^{\alpha+}$.
Let $\nu:[0,h]\to\RR$ be the Newton polygon of $M$.
For each $j\in\{1,\ldots,r\}$ let
$\nu_j:\RR\to\RR$ be the unique linear function
of slope $\lambda_j$ such that for all $t\in[0,h]$
we have
$$
\nu(t)=\max\{\nu_j(t)\mid 1\leq j\leq r\}.
$$
We define $\beta=(\beta_1,\ldots,\beta_r)\in\RR^r$ by $\beta_j=\nu_j(c)$.
\end{Notation}

The following result will be used in the
proof of Theorem \ref{Th-fDE-upper}.
Another application of it, to minimal Dieudonn\'e modules,
is given in Subsection \ref{Subsec-min-mod}.

\begin{Prop}
\label{Pr-fil}
We have $N^{\beta+}\subseteq pM$ as $W$-submodules of $N$.
\end{Prop}

We will deduce this from a description of $N^{\beta+}$
in terms of certain auxiliary valuations $\tilde w_j$
on $N$. Let $\Ical=\{0,1,\ldots,h-1\}$.
We recall that $(F^{c-i}z)_{i\in \Ical}$ is a
${W_\QQ}$-basis of $N$; see Lemma \ref{Le-present} (a).
For $j\in\{1,\ldots,r\}$ let $\tilde w_j$ be the valuation on $N$
which for $x=\sum_{i\in \Ical}e_iF^{c-i}z\in N$
with $e_i\in {W_\QQ}$ is given by
$$
\tilde w_j(x)=\min\{v(e_i)+(c-i)\lambda_j\mid i\in \Ical\}.
$$
This is an $F$-valuation only when $M$ is isoclinic.
A different way to look at $\tilde w_j$ is the following one.
Let $\nu_{x,j}:\RR\to\RR$ be the maximal linear
function of slope $\lambda_j$ such that we have
$\nu_{x,j}(i)\leq v(e_i)$ for all $i\in \Ical$. Then
$$
\tilde w_j(x)=\nu_{x,j}(c).
$$
Let $\nu_x:[0,h-1]\to\RR$ be defined by
$$
\nu_x(t)=\max\{\nu_{x,j}(t)\mid 1\leq j\leq r\}.
$$
For $\alpha\in\RR^r$ let
$$
\tilde N^{\alpha+}=\{x\in N\mid \tilde w_j(x)>\alpha_j
\text{ for } 1\leq j\leq r\}.
$$

\begin{Lemma}
\label{Le-fil}
We have $\tilde N^{\beta+}\subseteq pM$ as $W$-submodules of $N$.
\end{Lemma}

\begin{proof}
Let $x\in N$. As $\beta_j=\nu_j(c)$
we see that $x\in \tilde N^{\beta+}$ if and only if we have $\nu_{x,j}>\nu_j$
for all $1\leq j\leq r$. In other words, $x\in \tilde N^{\beta+}$ if and
only if $v(e_i)>\nu(i)$ for all $i\in \Ical$.
The last condition implies that $x\in pM$.
\end{proof}

By Lemma \ref{Le-fil} the following description
of $N^{\beta+}$ implies Proposition \ref{Pr-fil}.

\begin{Prop}
\label{Pr-fil+}
We have $N^{\beta+}=\tilde N^{\beta+}$ as $W$-submodules of $N$.
\end{Prop}

As the following lemma shows, 
the inclusion $\tilde N^{\beta+} \subseteq N^{\beta+}$ does not
depend on the specific choice of $\beta$.

\begin{Lemma}
\label{Le-ww} 
The following two properties hold:

\smallskip
(a) Let $j\in\{1,\ldots,r\}$. 
For $x\in N$ with projection $x_j\in N_j$ we have
$w_j(x_j)\geq\tilde w_j(x)$.

\smallskip
(b) For $\alpha\in\RR^r$ we have
$\tilde N^{\alpha+}\subseteq N^{\alpha+}$.
\end{Lemma}

\begin{proof}
If $x=F^{c-i}z$ with $i\in \Ical$,
then $w_j(x_j)=(c-i)\lambda_j=\tilde w_j(x)$.
As $w_j$ is a valuation and as
the ${W_\QQ}$-basis $(F^{c-i}z)_{i\in \Ical}$
of $N$ is valuative for $\tilde w_j$ we get (a).
Clearly (a) implies (b).
\end{proof}

The opposite inclusion $N^{\alpha+}\subseteq\tilde N^{\alpha+}$
does not hold in general.
To prove Proposition \ref{Pr-fil+} we need
conditions on $x$ under which equality
holds in Lemma \ref{Le-ww} (a).
For $x=\sum_{i\in \Ical} e_iF^{c-i}z$ with $e_i\in {W_\QQ}$
let $\Ical_j(x)\subseteq \Ical$ be the (non-empty)
set of those indices $i$ with
$v(e_i)=\nu_{x,j}(i)$ and let $s_j(x)$ be
the difference between the maximal and minimal elements
of $\Ical_j(x)$. Similarly let $\Lambda_j(x)\subseteq[0,h-1]$
be the closed interval of all $t$ with $\nu_{x,j}(t)=\nu_x(t)$
and let $s_j'(x)$ be the length of $\Lambda_j(x)$. 
We have $\Ical_j(x)\subseteq\Lambda_j(x)$; thus 
$s_j(x)\leq s_j'(x)$.

\begin{Lemma}
\label{Le-ww+}
Let $j\in\{1,\ldots,r\}$.
Let $x\in N$ and let $x_j\in N_j$ be its projection.
If $s_j(x)<h_j$, then $w_j(x_j)=\tilde w_j(x)$.
\end{Lemma}

\begin{proof}
As $w_j(x_j)\geq\tilde w_j(x)$, to prove that
$w_j(x_j)=\tilde w_j(x)$ we can replace $x$ by an arbitrary element $x'\in N$
with $\tilde w_j(x-x')>\tilde w_j(x)$. Using this and the inequality $s_j(x)<h_j$, we can assume that
for some integer $i_0$ the element $x$ lies in the $W_\QQ$-vector subspace of $N$
spanned by the finite set
$\Upsilon_{i_0,j}=\{F^{c-i}z\mid i_0\leq i\leq i_0+h_j-1\}$.
If $x\in\Upsilon_{i_0,j}$, then $w_j(x_j)=\tilde w_j(x)$.
Thus it suffices to show that $\Upsilon_{i_0,j}$ projects to a valuative
$W_\QQ$-basis of $N_j$ for $w_j$. For $i_0=c-c_j$ this is true
by Lemma \ref{Le-val-basis}. The general case follows because
the operators $F$ and $F^{-1}$ on $N_j$ preserve valuative
$W_\QQ$-bases as $w_j$ is an $F$-valuation.
\end{proof}

\begin{proof}[Proof of Proposition \ref{Pr-fil+}]
We know that $\tilde N^{\beta+}\subseteq N^{\beta+}$,
see Lemma \ref{Le-fil} (b). Thus to prove Proposition \ref{Pr-fil+}, it suffices to show that the assumption that there exists an element $x\in N^{\beta+}\setminus\tilde N^{\beta+}$ leads to a contradiction.

As $x\not\in\tilde N^{\beta+}$ there exists at least one
index $j\in \{1,\ldots,r\}$ such that $\tilde w_j(x)\leq\beta_j$.
Choose a maximal chain
$$
\Jcal=\{j_1,j_1+1,\ldots,j_2\}\subseteq \{1,\ldots,r\}
$$
such that $\tilde w_j(x)\leq\beta_j$ for all $j\in \Jcal$. Let
$[a,b]\subseteq[0,h]$ be the closed interval of all $t$
with $\nu(t)\in\{\nu_j(t)\mid j\in \Jcal\}$, and let
$[a',b']\subseteq[0,h-1]$ be the closed interval
of all $t$ with $\nu_x(t)\in\{\nu_{x,j}(t)\mid j\in \Jcal\}$.
In other words, $[a,b]$ (resp.\ $[a',b']$) is the
maximal interval where the slopes of $\nu$ (resp.\ of $\nu_x$)
lie in the set $\{\lambda_j\mid j\in \Jcal\}$.
We claim that the following implications hold:
\begin{gather*}
j_1>1\Rightarrow a'>a \\
j_2<r\Rightarrow b'<b.
\end{gather*}
As the proofs are similar, we will only check here the first implication. 
Assume that the implication does not hold i.e., we have $j_1>1$ and $a'\leq a$.
Then $j_0=j_1-1$ lies in $\{1,\ldots,r\}\setminus \Jcal$.
We have $\nu_{j_0}(a)=\nu_{j_1}(a)$ and
$\nu_{x,j_0}(a')=\nu_{x,j_1}(a')$, which implies that
$\nu_{x,j_0}(a)\leq\nu_{x,j_1}(a)$ as $a'\leq a$. Thus
we compute
\begin{align*}
\tilde w_{j_0}(x)-\beta_{j_0}
& {} = (\nu_{x,j_0}-\nu_{j_0})(c)
= (\nu_{x,j_0}-\nu_{j_0})(a) \\
& {} \leq (\nu_{x,j_1}-\nu_{j_1})(a)
= (\nu_{x,j_1}-\nu_{j_1})(c)
= \tilde w_{j_1}(x)-\beta_{j_1}
\leq 0
\end{align*}
which contradicts the maximality of $\Jcal$. This proves our claim.

We note that $j_1=1$ implies $a'=a=0$, while
$j_2=r$ implies $b'=h-1$ and $b=h$. Thus in all
cases we have $a'\geq a$ and $b'<b$, and hence
$$
b-a>b'-a'.
$$

On the other hand, as $x\in N^{\beta+}$,
for $j\in \Jcal$ we have
$$
\tilde w_j(x)\leq\beta_j<w_j(x_j).
$$
From this and Lemma \ref{Le-ww+} we get that $s_j'(x)\geq s_j(x)\geq h_j$.
Thus we get
$$
b-a=\sum_{j\in \Jcal}h_j\leq\sum_{j\in \Jcal}s_j'(x)=b'-a'.
$$
Contradiction.
This ends the proof of Proposition \ref{Pr-fil+}
(and thus also of Proposition \ref{Pr-fil}).
\end{proof}

\subsection{Minimal Dieudonn\'e modules}
\label{Subsec-min-mod}

Following Oort \cite{Oo3}, a Dieudonn\'e module $M$ over $k$ is
called minimal if $\End_\EE(M)$ is a maximal order in
$\End_\DD(M_\QQ)$ i.e., $M$ is the Dieudonn\'e module of a
minimal $p$-divisible group in the sense recalled in Section \ref{Se-prelim}.
A Dieudonn\'e module is minimal if and only if it
is a direct sum of isoclinic minimal Dieudonn\'e modules. 
In the isoclinic case we have the following
characterization of  minimality; see also \cite[Sect.~3]{Yu}.

\begin{Prop}
\label{Pr-isocl-min}
Let $M$ be an isoclinic Dieudonn\'e module of slope $\lambda$.
Then the following three statements are equivalent.

\smallskip
(a) {}
$M$ is minimal.

\smallskip
(b) {}
If $\Phi\in\DD$ satisfies $\ww_\lambda(\Phi)\geq 0$,
then $\Phi(M)\subseteq M$.

\smallskip
(c) {}
For some $F$-valuation $w$ on $N=M_\QQ$ of slope $\lambda$
we have $M=N^{w\geq 0}$.
\end{Prop}

\noindent
We begin with a special case; see also \cite[Subsects.~5.3-5.6]{dJO}.

\begin{Lemma}
\label{Le-simple-min}
If $N$ is a simple $\DD$-module of slope $\lambda$, then the statements (a) and (c) 
of Proposition \ref{Pr-isocl-min} are equivalent.
\end{Lemma}

\begin{proof}
Let $\Gamma=\End_\DD(N)$. Let $w$ be an $F$-valuation
of slope $\lambda$ on $N$. We note that $w$ is unique up to
adding a constant, see Lemma \ref{Le-F-val}. Thus
for each $\varphi\in\Gamma$ there exists a
$\tilde v(\varphi)\in\RR\cup\{\infty\}$ such that
$w(\varphi x)=w(x)+\tilde v(\varphi)$ for all $x\in N$.
Then $\tilde v$ is the unique valuation on the division
algebra $\Gamma$ that
extends the $p$-adic valuation on $\QQ_p$. The maximal
order in $\Gamma$ is $\Gamma_0=\{\varphi\in\Gamma\mid
\tilde v(\varphi)\geq 0\}$.

In view of these remarks, (c)$\Rightarrow$(a) is clear.
We prove (a)$\Rightarrow$(c). Let $h=\dim(N)$. We
choose $w$ such that $\ZZ\subseteq w(N)$. The $k$-vector space $gr^\alpha N$ is $1$-dimensional if $h\alpha\in\ZZ$,
and it is $0$ otherwise. In particular, $w(N)=(1/h)\ZZ$.
Let $\pi\in\Gamma_0$ be a generator of the maximal ideal,
which means that $\tilde v(\pi)=1/h$. As $M$ is stable
under $\pi$, the subset $w(M)$ of $w(N)$ takes the form
$\{i/h\mid i\in\ZZ,\;i\geq i_0\}$ for some integer $i_0$.
By replacing $w$ with $w-{i_0/h}$ we can assume that
$i_0=0$. It follows easily that $M=N^{w\geq 0}$.
\end{proof}

\begin{proof}[Proof of Proposition \ref{Pr-isocl-min}]
Let $\lambda=l/n$ with coprime integers $l,n$ and $n\geq 1$.
Let $\Phi_0=F^np^{-l}\in\DD$ and choose
$\Phi=F^ap^m\in\DD$ such that $\ww_\lambda(\Phi)=1/n$.
The element $\Phi$ is unique up to multiplication by an
integral power of $\Phi_0$. Let $q=p^n$. Let $\mu=\dim(N)/n$
be the multiplicity of the $\DD$-module $N$.
First we show that (b) implies the existence
of a $W$-basis of $M$ of the form
$$
\Upsilon_1=(\Phi^ix_j)_{0\leq i<n,\;1\leq j\leq\mu}
$$
such that each $x_j\in M$ satisfies the equation $\Phi_0(x_j)=x_j$.
Indeed, let $\Pi=\{x\in M\mid\Phi_0(x)=x\}$.
This is a $W(\FF_q)$-submodule of $M$.
As $\Phi_0$ has slope zero and preserves $M$ by the assumption (b),
we have $M=\Pi\otimes_{W(\FF_q)}W$.
As $\Phi^n:\Pi\to\Pi$ is multiplication by $p$,
the quotient $\Pi/\Phi\Pi$ is an $\FF_q$-vector space
of dimension $\mu$.
We choose elements $x_1,\ldots,x_\mu\in\Pi$ which project
to an $\FF_q$-basis of $\Pi/\Phi\Pi$. Then for each $i\geq 0$,
$(\Phi^ix_j)_{1\leq j\leq\mu}$ projects to an $\FF_q$-basis of
$\Phi^i\Pi/\Phi^{i+1}\Pi$.
We conclude that $\Upsilon_1$ is a $W$-basis of $M$.

The implication (c)$\Rightarrow$(b) is clear.
We prove (b)$\Rightarrow$(c). Let $\Upsilon_1$ be as above.
There exists a unique $F$-valuation $w$ of slope $\lambda$ on $N$
such that $w(x_j)=0$ for $1\leq j\leq \mu$
and such that the $W$-basis $\Upsilon_1$ is valuative for $w$.
As $w(\Phi^ix_j)=i/n$ lies in the interval $[0,1)$
when $0\leq i<n$, property (c) follows easily.

We prove (b)$\Rightarrow$(a). Let $\Upsilon_1$ be as above
and let $M_j=\DD^{\ww_\lambda\geq 0}x_j$ for $1\leq j\leq\mu$.
Then as $\Phi_0(x_j)=x_j$, it follows that $M=M_1\oplus\cdots\oplus M_\mu$ is a direct sum decomposition
into pairwise isomorphic simple Dieudonn\'e modules.
Hence $\End_\EE(M)$ is a matrix algebra over $\End_\EE(M_1)$, and
(a) follows by Lemma \ref{Le-simple-min}.

Finally, we prove (a)$\Rightarrow$(c). Let $N_1$ be a simple
constituent of $N$ and let $\Gamma=\End_\DD(N_1)$. 
As each maximal
order in $\End_\DD(N)$ is isomorphic to the matrix algebra over
the maximal order $\Gamma_0$ of $\Gamma$ we see that $M$ 
is the direct sum of simple Dieudonn\'e modules.
Hence (c) follows from Lemma \ref{Le-simple-min} again.
\end{proof}

The next proposition is proved as well in \cite[Lemma 4.2]{Yu}.

\begin{Prop}
\label{Pr-min-min-ex}
Let $M$ be a Dieudonn\'e module and let $N=M_\QQ$. There exists a
minimal minimal Dieudonn\'e module $M_+$ with $M\subseteq
M_+\subset N$ i.e., for every minimal Dieudonn\'e module $M'$ with
$M\subseteq M'\subset N$ we have $M_+\subseteq M'$.
\end{Prop}

\begin{proof}
If $N$ is isoclinic, then by Proposition \ref{Pr-isocl-min} every minimal
$M'$ as above takes the form $M'=N^{w\geq 0}$ for some
$F$-valuation $w$ on $N$ of slope $\lambda$
with $w(M)\geq 0$. Hence in the isoclinic
case the assertion follows from Lemma \ref{Le-min-val}. In
general, let $N=\bigoplus_{j=1}^r N_j$ be the direct sum decomposition into isoclinic
components and let $M_j$ be the image of $M\to N_j$. As each
$M'$ as above is a direct sum $M'=\bigoplus_{j=1}^r M'_j$ with
$M_j\subseteq M'_j\subset N_j$, we have $M_+=\bigoplus_{j=1}^r (M_j)_+$.
\end{proof}

\noindent By duality there exists also a maximal minimal
Dieudonn\'e module $M_-\subseteq M$.

\begin{Lemma}
\label{Le-min-funct}
A homomorphism of Dieudonn\'e modules $f:M'\to M$ induces
homomorphisms $f_-:M'_-\to M_-$ and $f_+:M'_+\to M_+$.
\end{Lemma}

\begin{proof}
By duality it suffices to show that $f(M'_-)\subseteq M_-$.
We can assume that $M'$ and $M$ are isoclinic of the same
slope $\lambda$.
Using Proposition \ref{Pr-isocl-min} it is easy to see that
$M_-$ is the set of all elements $x\in M_\QQ$ such that we have $\Phi x\in M$
for each $\Phi\in\DD$ with $\ww_\lambda(\Phi)\geq 0$,
and the analogous statement holds for $M'_-$. As $f_\QQ:M_\QQ\to M'_\QQ$ is a
$\DD$-linear map it follows that $f$ maps $M'_-$
to $M_-$.
\end{proof}

If $M=M_b\oplus M_o$ is the unique decomposition such that
$M_b$ is bi-nil\-potent and $M_o$ has integral slopes, then
$M_\pm=M_{b\pm}\oplus M_o$. 
We have the following explicit descriptions of $M_+$ and $M_-$ in the case when $a(M)=a(M_b)=1$.

\begin{Thm}
\label{Th-max-min}
Let $M$ be a bi-nilpotent Dieudonn\'e module and let
$M_-\subseteq M\subseteq M_+$ be the 
minimal Dieudonn\'e modules considered above.
We assume that $a(M)=1$ and we use Notation \ref{Notat}.

\smallskip
(a) We have $M_+=N^{\underline 0}$, where
$\underline 0=(0,\ldots,0)\in\RR^r$.

\smallskip
(b) We have $M_-=p^{-1}N^{\beta+}$.
\end{Thm}

\noindent
By Proposition \ref{Pr-fil+} the $W$-module 
$N^{\beta+}$ has a $W$-basis consisting of easily
computable $W$-multiples 
of $F^{c-i}z$ for $i\in \Ical=\{0,\ldots,h-1\}$, which makes it
more explicit than the $W$-module $N^{\underline 0}$.

\begin{proof}[Proof of Theorem \ref{Th-max-min}]
Let $M_j$ be the image of $M\to N_j$. 
We recall that $w_j$ is the minimal $F$-valuation of slope
$\lambda_j$ on $N_j$ such that $w_j(M_j)\ge 0$.
This implies (a); see the proof of Proposition
\ref{Pr-min-min-ex}. 

The Dieudonn\'e module $p^{-1}N^{\beta+}$ is
minimal by Proposition \ref{Pr-isocl-min} and contained in $M$ by Proposition \ref{Pr-fil}. Thus we have inclusions
$$
p^{-1}N^{\beta+}\subseteq M_-\subseteq M
\subseteq M^+=N^{\underline 0}.
$$

For a $W$-module $A$, let $A^\vee=\Hom_W(A,W)$. If $A$ has finite length, let
$\ell(A)$ be its length. 
Consider 
$$
\ell_1 =\ell(M/p^{-1}N^{\beta+}) 
\qquad \text{and} \qquad
\ell_2 =\ell (N^{\underline 0}/p^{-1}N^{\beta+}).
$$
We claim that $\ell_2=2\ell_1$. This implies that $\ell(M_+/M)\geq\ell(M/M_-)$ with equality if and only if
$M_-=p^{-1}N^{\beta+}$. The same reasoning applied to the dual
Dieudonn\'e module $M^\vee$ gives the opposite inequality $\ell(M/M_-)\geq\ell(M_+/M)$. Here we use that $a(M^\vee)=1$ and
$(M^\vee)_+=(M_-)^\vee$ and $(M^\vee)_-=(M_+)^\vee$. Thus (b)
follows from our claim.

It remains to show that $\ell_2=2\ell_1$.
Let $s$ be the multiplicity of $N$ (i.e., the sum of the multiplicities of the isoclinic direct factors $N_j$ of $N$). Then
$$
\ell_2+h-s= \ell(pN^{\underline 0}/N^{\beta+})+\ell(N^{\underline 0}/pN^{\underline 0})-\ell(N^\beta/N^{\beta+})=
\ell(N^{\underline 0}/N^\beta)=\sum_{j=1}^r\beta_jh_j.
$$

On the other hand, let $\rho$ be the ordinary Newton
polygon with the same endpoints as $\nu$ and let
$\varOmega\subseteq\RR^2$ be the compact set 
enclosed by $\nu$ and $\rho$. 
A $W$-basis of $pM$ is formed by the elements
$p^{n_i}F^{c-i}z$ for $0\le i<h$ where $n_i$ is
the minimal integer such that $n_i>\rho(i)$.
A $W$-basis of $\tilde N^{\beta+}$ is formed
by the elements $p^{m_i}F^{c-i}z$ for $0\le i<h$ 
where $m_i$ is the minimal integer with $m_i>\nu(i)$.
Thus $\ell(pM/\tilde N^{\beta+})$ is the number of
elements of $\ZZ^2$ which lie strictly above $\rho$
and on or below $\nu$. Hence the cardinality of
the finite set $\interior(\varOmega)\cap\ZZ^2$ 
is equal to $\ell(pM/\tilde N^{\beta+})-s+1$, 
which is equal to $\ell_1-s+1$ by Proposition
\ref{Pr-fil+}. The set $\partial\varOmega\cap\ZZ^2$ has $s+h$
elements. Hence the area of $\varOmega$ can be expressed in two ways
as follows.
$$
\int_0^h\nu(t)\,dt-d^2/2=\ell_1-s+1+\frac{s+h}2-1
=\ell_1+\frac{h-s}2.
$$

Finally, the function $g(t)=\nu(t)-\nu'(t)(t-c)$ is
well-defined for those $t\in[0,h]$ where $\nu$ is linear. Its
value is $g(t)=\beta_j$ if $\nu'(t)=\lambda_j$. Thus we get
$$
\sum_{j=1}^r\beta_jh_j=\int_0^h(\nu(t)-\nu'(t)(t-c))\,dt
=2\int_0^h\nu(t)\,dt-d^2
$$
by integration by parts. The last three displayed equations give
$\ell_2=2\ell_1$.
\end{proof}

The $p$-exponent of a finitely generated torsion $W$-module
$\bar M$ is the smallest non-negative integer $m$ such that we have
$p^mx=0$ for all $x\in\bar M$.

\begin{Lemma}
\label{Le-three-p-exp}
Let $M$ be a Dieudonn\'e module. Then the three $W$-modules
$M_+/M_-$, $M/M_-$, and $M_+/M$ have the same $p$-exponent.
\end{Lemma}

\begin{proof}
If $p^m$ annihilates $M_+/M$, then $p^mM_+\subseteq M$ and thus
we have $p^mM_+\subseteq M_-$ as $p^mM_+$ is
minimal. Similarly, if $p^m$ annihilates $M/M_-$, then we
have $M\subseteq p^{-m}M_-$ which implies that
$M_+\subseteq p^{-m}M_-$. 
\end{proof}

\begin{Cor}
\label{Co-max-min}
Let $M$ be a Dieudonn\'e module of Newton polygon $\nu$. Let $m$ be the $p$-exponent of $M_+/M_-$. We have $m\leq\lfloor\nu(c)\rfloor$ with equality if $a(M)=1$.
\end{Cor}

\begin{proof}
We can assume that $M$ is bi-nilpotent and non-zero.
If $a(M)=1$, then $m$ is the $p$-exponent of
$p^{-1}N^{\beta+}/N^{\underline 0}$
by Theorem \ref{Th-max-min} and thus we have $m=\max\{\lfloor\nu_j(c)\rfloor\mid 1\leq j\leq r\}=\lfloor\nu(c)\rfloor$ (cf.\ Lemma \ref{Le-p-exp}). If $a(M)\ge 2$, then by Lemma \ref{Le-three-p-exp}
it suffices to show that for every non-zero element $x\in M$ we have
$p^{\lfloor\nu(c)\rfloor}x\in M_-$. Let $M'=\EE x$ be the
Dieudonn\'e module generated by $x$. Let $\nu'$ be its
Newton polygon and let $c'$ be the dimension of $M'/FM'$.
As $a(M')=1$, we know that $p^{\lfloor\nu'(c')\rfloor}x\in M'_-$.
By Lemma \ref{Le-min-funct}, $M'_-$ is contained in $M_-$.
Thus $\nu'(c')\leq\nu(c)$ by Lemma \ref{5.24} below. 
Hence $p^{\lfloor\nu(c)\rfloor}x\in M_-$ as desired.
\end{proof}

\begin{Lemma}
\label{5.24}
Let $\nu$ and $\nu'$ be Newton polygons with endpoints $(c+d,d)$ and 
$(c'+d',d')$ (respectively). If the slopes of $\nu'$ form a subset of the slopes 
of $\nu$ (counted with multiplicities), then $\nu'(c') \le \nu(c)$.
\end{Lemma}

\begin{proof}
It is easy to see that the number $\nu(c)$ is invariant under duality in the sense
that $\nu(c)=\nu^\vee(d)$ if $\nu^\vee$ is defined by $\nu^\vee(x)=\nu(c+d-x)+x-d$ 
for $x\in[0,c+d]$. To prove the lemma, by induction it suffices to consider the
case where $\nu'$ arises from $\nu$ by deleting a line (i.e., an isoclinic part)
of some slope $\lambda$. Then at least one of the following holds:

\smallskip
(1) {} The slopes of $\nu'(x)$ for $x\le c'$ are less or equal to $\lambda$; 

\smallskip
(2) {} The slopes of $\nu'(x)$ for $x\ge c'$ are greater or equal to $\lambda$. 

\smallskip
\noindent
The passage to the duals of $\nu$ and $\nu'$ 
interchanges (1) and (2) and therefore we can assume that (1) holds. 
Then $\nu(x)=\nu'(x)$ for $x\le c$ and thus $\nu'(c')=\nu(c')\le\nu(c)$
 as $c'\le c$ and the function $\nu$ is increasing.
\end{proof}

\begin{proof}[Proof of Theorem \ref{Th-min-isog}]
Let $M$ be the covariant Dieudonn\'e module of the given
$p$-divisible group $D$ over $k$. Isogenies $f:D\to D_0$
with minimal $D_0$ correspond to minimal Dieudonn\'e
modules $M_0$ with $M\subseteq M_0\subseteq M_\QQ$ in such
a way that the $p$-exponents of $\Ker(f)$ and of $M_0/M$ coincide.
Hence Theorem \ref{Th-min-isog} follows from Corollary
\ref{Co-max-min} together with Lemma \ref{Le-three-p-exp}.
\end{proof}


\section{Values of isogeny cutoffs}
\label{Se-val-isog}

In this section we fix a $p$-divisible group $D$
over $k$ of dimension $d$ and codimension $c$ with Newton polygon $\nu$. 
We will prove Theorem \ref{Th-b-maximal} and 
list all possible values of the isogeny cutoff $b_D$ of $D$.

\begin{Lemma}
\label{Le-b-upper-prelim}
To prove Theorem \ref{Th-b-maximal} we can
assume that $D$ is connected with connected dual
and that $a_D=1$.
\end{Lemma}

\begin{proof}
There exists a $p$-divisible group
over $k[[t]]$ whose special fibre is $D$ and whose geometric generic fibre has $a$-number
at most $1$, cf.\ \cite[Prop.~2.8]{Oo1}.
By Theorem \ref{Th-going-down-b}, after replacing if needed $k$
by an algebraic closure of $k((t))$,
we can assume that $a_D\le 1$.
Let $D=D^{\circ}\times D^{\text{ord}}
$, where $D^{\circ}$ is connected
with connected dual and $D^{\text{ord}}$ is ordinary.
Let $\nu_0$ be the Newton polygon of $D^{\circ}$ and let
$c_0$ be the codimension of $D^{\circ}$.
We have $\nu(c)=\nu_0(c_0)$, $b_D=b_{D^{\circ}}$
(by Lemma \ref{Le-prelim} (c)), and $a_{D^{\circ}}=a_D\le 1$. 
If $a_{D^\circ}=0$ then $D^\circ$ is trivial.
Thus to prove Theorem \ref{Th-b-maximal} we can assume that $D=D^{\circ}$ and $a_D=1$.
\end{proof}

\begin{proof}[Proof of Theorem \ref{Th-b-maximal}]
Let $M$ be the covariant Dieudonn\'e module of $D$.
By Lemma \ref{Le-b-upper-prelim} we can assume
that $M$ is bi-nilpotent and that $a(M)=1$.
We write $M=\EE/\EE\Psi$ as in Lemma \ref{Le-present}
with $a_0=1$.

Let $D'$ be another $p$-divisible
group over $k$ with Dieudonn\'e module $M'$
and with Newton polygon $\nu'$.

First we show that $b_D\leq j(\nu)$.
Assume that $D[p^{j(\nu)}]\cong D'[p^{j(\nu)}]$;
we must show that $\nu=\nu'$.
As $j(\nu)\geq 1$, the $p$-divisible group $D'$ and its
 dual are connected and we have $a_{D'}=1$. Choose an element
$z'\in M'$ such that the class of $z$ maps to the class of
$z'$ under the isomorphism $M/p^{j(\nu)}M\cong M'/p^{j(\nu)}M'$,
and let $\Psi'z'=0$ be the associated relation given
by Lemma \ref{Le-present} (b) with $a_0=1$.
Then $\Psi'-\Psi\in p^{j(\nu)}\EE$.
Let us write $\Psi-\Psi'=\sum_{i=1}^h e_iF^{c-i}$.
By the definition of $j(\nu)$, as $\Psi'-\Psi\in p^{j(\nu)}\EE$ we have $v(e_i)\geq\nu(i)$ always and $v(e_i)>\nu(i)$ if $(i,\nu(i))$ is a breakpoint of $\nu$.
From this and Lemma \ref{Le-Newton} we get that $\nu'=\nu$.

Next we show that $b_D\geq j(\nu)$. If $j(\nu)=1$ this is clear.
Thus we can assume that $m=j(\nu)-1>0$.
If $m<\nu(c)$, we take $\Psi'=\Psi+p^m$.
If $m\geq\nu(c)$, then we have $m=\nu(c)$ and $(c,\nu(c))$
is a breakpoint of $\nu$, in particular $v(a_c)=\nu(c)=m$.
In this case, we take $\Psi'=\Psi-a_c$.
We define $D'$ by $M'=\EE/\EE\Psi'$.
Then $D[p^m]\cong D'[p^m]$. From Lemma \ref{Le-Newton} we get that $\nu'\neq\nu$. Thus $m<b_D$.
\end{proof}

\begin{Prop}
\label{Pr-b-values}
If $m$ is an integer with $1\leq m\leq j(\nu)$, then
there exists a $p$-divisible group $D$ over $k$ with
Newton polygon $\nu$ and $b_D=m$.
\end{Prop}

\begin{proof}
The assertion is true
for $m=j(\nu)$, cf.\ Theorem \ref{Th-b-maximal}. It is also true for $m=1$, as for a minimal $p$-divisible group $D_0$ we have $n_{D_0}=1$ (see either \cite[Thm.~1.2]{Oo3} or \cite[Thm.~1.6]{Va3}) and thus also $b_{D_0}=1$.
As two $p$-divisible groups over $k$ of the same
Newton polygon $\nu$ can be linked by a chain of isogenies
with kernels annihilated by $p$, the proposition
follows from the next lemma.
\end{proof}

\begin{Lemma}
\label{Le-b-change}
Let $g:D\to E$ be an isogeny of $p$-divisible groups over
$k$ such that the kernel of $g$ is annihilated by $p$.
Then $|b_D-b_E|\leq 1$.
\end{Lemma}

\begin{proof}
As $pg^{-1}$ is an isogeny $E\to D$ with kernel annihilated by
$p$, by symmetry it suffices to show that $b_E\geq b_D-1$.
Let $m$ be an integer with $0<m<b_D$. This means that there exists
a $p$-divisible group $D'$ over $k$ which is not isogenous to $D$
 and for which there exists
an isomorphism $u:D[p^m]\cong D'[p^m]$. Let $E'=D'/u(\Ker g)$.
Then $E'$ is not isogenous to $E$ and $u$ induces an isomorphism
$E[p^{m-1}]\cong E'[p^{m-1}]$. Thus $m-1<b_E$. Therefore $b_E\ge b_D-1$.
\end{proof}

\begin{Remark}
\label{Re-b-bound}
Here is another approach to
bound $b_D$ from above. By Theorem \ref{Th-min-isog}
there exists an isogeny $D\to D_0$ with kernel annihilated
by $p^{\lfloor\nu(c)\rfloor}$ and with $D_0$ minimal.
As $b_{D_0}=n_{D_0}=1$,
either Lemma \ref{Le-b-change} or \cite[Lemma 2.9]{NV2} gives
$b_D\leq 1+\lfloor\nu(c)\rfloor$.
This estimate is
equivalent to the upper bound in Theorem \ref{Th-b-maximal}
except when $\nu(c)\in\ZZ$ and $\nu$ is linear at $c$;
then it is off by $1$. 
\end{Remark}

\begin{Remark}
We assume that either $\nu(c)\notin\ZZ$ or $\nu$ is not linear at $c$; equivalently, we have $j(\nu)=\lfloor \nu(c)\rfloor+1$. 
In this case we have the following refinement of
Proposition \ref{Pr-b-values}.
By Theorem \ref{Th-min-isog} there exists a chain $D_1\leftarrow\cdots\leftarrow D_{j(\nu)}$ of isogenies of $p$-divisible groups of Newton polygon $\nu$, where $D_1$ is minimal, where $a_{D_{j(\nu)}}=1$ and thus $b_{D_{j(\nu)}}=j(\nu)$ by Theorem \ref{Th-b-maximal}, and where the kernel of each $D_i\leftarrow D_{i+1}$ with $i\in\{1,\ldots,j(\nu)-1\}$ is annihilated by $p$. Lemma \ref{Le-b-change} implies that we have $b_{D_i}=i$ for all $i\in\{1,\ldots,j(\nu)\}$.
\end{Remark}

\section{A variant of homomorphism numbers}

Let $D$ and $E$ be $p$-divisible groups over the
algebraically closed field $k$. As suggested by the
results of \cite{GV} (see Theorem \ref{Th-f-n} below) 
we consider the following variant 
of the homomorphism numbers $e_{D,E}$.

\begin{Lemma}[{\cite[Subsect.~6.1]{GV}}]
\label{Le-f}
There exists a non-negative integer $f_{D,E}$ such that for positive
integers $m\geq n$ the restriction homomorphism 
$$
\tau_{m,n}:\Hom(D[p^m],E[p^m])\to\Hom(D[p^n],E[p^n])
$$
has finite image if and only if $m\geq n+f_{D,E}$.
\end{Lemma}

\begin{proof}
As $\Hom(D,E)$ is a finitely generated $\ZZ_p$-module, its
image in the $p^n$-torsion group $\Hom(D[p^n],E[p^n])$ is
finite for each $n$. By Lemma \ref{Le-eDE} it follows that 
for each $n$ there exists an $m$ such that $\tau_{m,n}$ has finite
image. We have to show that the minimal such $m$ takes the
form $m=n+f$ where $f$ does not depend on $n$. This
follows from the lower exact sequence in \eqref{Eq-eED} in the proof of
Lemma \ref{Le-eDE} as for the numbers $e_{D,E}(n)$.
\end{proof}

\begin{Defn}
We call $f_{D,E}$ the \emph{coarse homomorphism number} of $D$ and $E$. 
The \emph{coarse endomorphism number} of $D$ is $f_D=f_{D,D}$.
\end{Defn}

Using this notion we can state \cite[Cor.~2 (b)]{GV} as follows.

\begin{Thm}[{\cite{GV}}]
\label{Th-f-n}
If $D$ is not ordinary, then we have $f_D=n_D$.
\end{Thm}

\begin{Remark}
\label{Re-f-n}
We note that \cite{GV} also gives a similar interpretation
of $f_{D,E}$ in general. Namely, let 
$n_{D,E}$ be the minimal non-negative integer $m$ such that
the truncation map $\Ext^1(D,E)\to\Ext^1(D[p^m],E[p^m])$
is injective. Then we have $f_{D,E}=n_{D,E}$, cf.\ 
\cite[Subsect.~6.1 (iii)]{GV}. The resulting 
equality $n_D=n_{D,D}$ if $D$ is not ordinary
seems not to be obvious from the definitions. 
\end{Remark}

The following is the first and easiest of three related
inequalities.

\begin{Prop}
\label{Pr-f-e}
We have $f_{D,E}\leq e_{D,E}.$
\end{Prop}

\begin{proof}
This is immediate from the definitions and the finiteness
of the image of the reduction homomorphism
$\Hom(D,E)\to\Hom(D[p^n],E[p^n])$.
\end{proof}

The following analogue of Lemma \ref{Le-prelim} is easily checked.

\begin{Lemma}
\label{Le-prelim-f}
Let $D=D^\text{ord}\times D^\circ$ and 
$E=E^\text{ord}\times E^\circ$ be the canonical
decompositions such that $D^\text{ord}$ and $E^\text{ord}$
are the maximal ordinary subgroups of $D$ and $E$ (respectively).
Then $f_{D,E}=f_{D^\circ,E^\circ}$. \qed
\end{Lemma}

\begin{Lemma}
\label{Le-perm-f}
For each algebraically closed field 
$\kappa\supseteq k$ we have 
$f_{D,E}=f_{D_\kappa,E_\kappa}.$
\end{Lemma}

\begin{proof}
For positive integers $m\geq n$ we have $f_{D,E}> m-n$
if and only if the image of $\uHom(D[p^m],E[p^m])\to\uHom(D[p^n],E[p^n])$ has positive dimension. 
This property is invariant under the base change from $k$ to $\kappa$.
\end{proof}

\subsection{Semicontinuity of coarse homomorphism numbers}

We will show that the coarse homomorphism numbers $f_{D,E}$
are lower semicontinuous in families of $p$-divisible 
groups with constant Newton polygon.
The following characterization of $f_{D,E}$ is taken
from \cite{GV}. 

\begin{Lemma}[{\cite[Subsect.~6.1 (i) and (ii)]{GV}}]
\label{Le-gamma-f}
For $m\in\NN$ let
$$\gamma_{D,E}(m)=\dim(\uHom(D[p^m]),E[p^m]).$$
We have $\gamma_{D,E}(m)\leq\gamma_{D,E}(m+1)$
with equality if and only if $m\geq f_{D,E}$.
\end{Lemma}

\begin{proof}
Let ${\uH}{}_m=\uHom(D[p^m]),E[p^m])$.
We have an exact sequence of algebraic groups
$
0\to {\uH}{}_m\xrightarrow{\iota_m} {\uH}{}_{m+1}\xrightarrow{\tau_{m+1,1}} {\uH}{}_1
$, cf.\ the proof of Lemma \ref{Le-eDE}. Thus the lemma follows from the definition of $f_{D,E}$,
see Lemma \ref{Le-f}.
\end{proof}

\begin{Defn}
The stable value of $\gamma_{D,E}$ is denoted
$s_{D,E}=\gamma_{D,E}(f_{D,E})$. If $D=E$ we write
$\gamma_{D,D}(m)=\gamma_D(m)$ and $s_{D,D}=s_D=\gamma_D(f_D)$.
\end{Defn}

The semicontinuity of $f_{D,E}$ relies on
the following result of \cite{Va2}.

\begin{Thm}[{\cite[Thm.~1.2 (f)]{Va2}} and {\cite[Rm.~4.5]{GV}}]
\label{Th-gamma-isog} If $D$ and $D'$ are isogenous $p$-divisible
groups over $k$, then we have $s_D=s_{D'}$.
\end{Thm}

As earlier, if $\scrD$ and $\scrE$ are $p$-divisible groups over an
$\FF_p$-scheme $S$, we define functions 
$f_{\scrD,\scrE},\gamma_{\scrD,\scrE}(m),s_{\scrD,\scrE},s_{\scrD}=s_{\scrD,\scrD}:S\to\NN$ 
by $f_{\scrD,\scrE}(s)=f_{\scrD_{\bar s},\scrE_{\bar s}}$, etc., 
where $\scrD_{\bar s}$ and $\scrE_{\bar s}$ are 
the geometric fibres of $\scrD$ and $\scrE$ over $s\in S$.

\begin{Thm}
\label{Th-semicont-f}
Let $\scrD$ and $\scrE$ be $p$-divisible groups of constant
Newton polygon over an $\FF_p$-scheme $S$.

\smallskip
(a) The function $s_{\scrD,\scrE}$ is locally constant on $S$.

\smallskip
(b) For $m\in\NN$, the set $U_{f_{\scrD,\scrE}} =\{s\in S\mid f_{\scrD,\scrE}(s)\leq m\}$
is closed in $S$.
\end{Thm}

\begin{proof}
We have $s_{\scrD\oplus\scrE}=s_{\scrD}+s_{\scrE}+s_{\scrD,\scrE}+s_{\scrE,\scrD}$ as functions on $S$. The functions $s_{\scrD\oplus\scrE}$,
$s_\scrE$, and $s_\scrD$ are constant by Theorem \ref{Th-gamma-isog}.
Locally in $S$, for sufficiently large $m$ we have
$s_{\scrD,\scrE}=\gamma_{\scrD,\scrE}(m)$ and 
$s_{\scrE,\scrD}=\gamma_{\scrE,\scrD}(m)$. 
These functions are upper semicontinuous.
Hence $s_{\scrD,\scrE}=s_{\scrD\oplus\scrE}-s_{\scrD}-s_{\scrE}-s_{\scrE,\scrD}$ is upper and lower semicontinuous,
thus locally constant, which proves (a).

To prove (b) we note that for $s\in S$ we have
$(s_{\scrD,\scrE}-\gamma_{\scrD,\scrE}(m))(s)\geq 0$ with 
equality if and only if $s\in U_{f_{\scrD,\scrE}} $, see Lemma \ref{Le-gamma-f}.
As $s_{\scrD,\scrE}$ is locally constant and $\gamma_{\scrD,\scrE}(m)$
is upper semicontinuous we get that $U_{f_{\scrD,\scrE}} $ is closed. 
\end{proof}

\begin{Remark}\label{7.13}
We have $U_{n_{\scrD}}=U_{f_{\scrD,\scrD}}$, cf.\ Theorem \ref{Th-f-n}. Thus Theorem \ref{Th-semicont} (b) also follows from Theorem \ref{Th-semicont-f} (b). 
\end{Remark}

\section{Homomorphisms of truncated $BT$ groups}
\label{Se-Hom}

In this section, we prove Theorem \ref{Th-n-ell-e}, 
but we work more generally with homomorphisms instead of endomorphisms.

\subsection{The level torsion $\ell_{D,E}$}
\label{onell}

Let $W(k)$ and $\sigma$ be as in Section \ref{Complem}. Let $D$ and $E$ be two $p$-divisible groups over $k$.
Let $(M,F,V)$ and $(L,F,V)$ be the covariant Dieudonn\'e modules of $D$ and $E$ (respectively). Let $H=H_{D,E}=\Hom_{W(k)}(M,L)$. Let $(H[\frac 1p],F)$ be the 
$F$-isocrystal defined by the rule 
$$
F(\flat)=F\circ\flat\circ F^{-1}= V^{-1}\circ \flat\circ V.
$$
We consider the direct sum decomposition into $W(k)[1/p]$-vector spaces
$$H[\tfrac 1p]=N_+\oplus N_0\oplus N_-$$
which is invariant under $F$ and such that all slopes of $(N_+,F)$ are positive, all slopes of $(N_0,F)$ are $0$, and all slopes of $(N_-,F)$ are negative. 

\begin{Defn}
\label{Def-ell}
Let $O_+\subseteq H\cap N_+$, $O_0\subseteq H\cap N_0$, and
$O_-\subseteq H\cap N_-$
be the maximal $W(k)$-submodules such that $F(O_+)\subseteq O_+$,
$F(O_0)=O_0$, and $F^{-1}(O_-)\subseteq O_-$. The \emph{level module}
of $D$ and $E$ is defined as
$$ O=O_{D,E}=O_+\oplus O_0\oplus O_-.$$
If $D=E$ we write $O_{D,D}=O_D$.
\end{Defn}

For a more explicit description, let $A_0=\{z\in H\mid F(z)=z\}$, 
which we identify with the free $\ZZ_p$-module $\Hom(D,E)$.
We have
$$O_+=\{z\in N_+\mid F^t(z)\in H\quad\forall t\in\NN\},$$
$$O_0=A_0\otimes_{\ZZ_p} W(k)=\bigcap_{t\in\NN} F^t(H\cap N_0)
=\bigcap_{t\in\NN} F^{-t}(H\cap N_0),$$
$$O_-=\{z\in N_-\mid F^{-t}(z)\in H\quad\forall t\in\NN\}.$$

\begin{Lemma}
The $W(k)$-module $O$ is a lattice of $H[\frac 1p]$.
\end{Lemma}

\begin{proof}
As all slopes of $(N_+,F)$ are positive, for each $z\in N_+$ the sequence
$(F^t(z))_{t\in\NN}$ of elements of $N_+$ converges to $0$ in the $p$-adic topology.
Therefore there exists $s\in\NN$ such that $p^sz\in O_+$. Thus we have
$O_+[\frac 1p]=N_+$. As $O_+$ is a $W(k)$-submodule of the finitely generated
$W(k)$-module $H$, we conclude that $O_+$ is a lattice of $N_+$. A similar argument
shows that $O_0$ and $O_-$ are lattices of $N_0$ and $N_-$ (respectively). Thus $O$ is a lattice of $H[\frac 1p]$.
\end{proof}

\begin{Defn}
The {\it level torsion} of $D$ and $E$ is the smallest non-negative integer $\ell_{D,E}$ such that we have
$$p^{\ell_{D,E}}\Hom_{W(k)}(M,L)\subseteq O\subseteq \Hom_{W(k)}(M,L).$$
If $D=E$ we write $\ell_{D,D}=\ell_D$.
\end{Defn}

If $D$ and $E$ are isoclinic, then this definition of $\ell_{D,E}$ coincides with the one in \cite[Def.~4.1 (c)]{Va3}. 
The level torsion is symmetric:

\begin{Lemma}
\label{Le-ell-symm}
We have $\ell_{D,E}=\ell_{E,D}=\ell_{D^\vee,E^\vee}$.
\end{Lemma}

\begin{proof}
Let $\tilde O=\tilde O_{D,E}\subseteq H[\frac 1p]$ be the 
minimal $W(k)$-submodule which contains $H$ and which takes 
the form $\tilde O=\tilde O_+\oplus\tilde O_0\oplus\tilde O_-$
with $\tilde O_+\subseteq N_+$, $\tilde O_0\subseteq N_0$,
and $\tilde O_-\subseteq N_-$, such that 
$F(\tilde O_+)\subseteq\tilde O_+$, $F(\tilde O_0)=\tilde O_0$,
and $F^{-1}(\tilde O_-)\subseteq\tilde O_-$. Let $\tilde\ell_{D,E}\in\NN$
be the minimal number such that 
$p^{\tilde\ell_{D,E}}\tilde O\subseteq H$.
First we show that $\ell_{D,E}=\tilde\ell_{D,E}$.
We have $H\subseteq p^{-\ell_{D,E}}O$, thus
$O\subseteq H\subseteq\tilde O\subseteq p^{-\ell_{D,E}}O$ 
by the minimality of $\tilde O$, and therefore
$\tilde\ell_{D,E}\leq\ell_{D,E}$. Similarly we have
$p^{\tilde\ell_{D,E}}\tilde O\subseteq O\subseteq H\subseteq \tilde O$ by the maximality of $O$, and therefore $\ell_{D,E}\leq\tilde\ell_{D,E}$.

Let $F_{D,E}=F$ be the Frobenius of $H_{D,E}[\frac 1p]$.
For a $W(k)$-module $A$ we write $A^\vee=\Hom_{W(k)}(A,W(k))$.
The Dieudonn\'e modules of $D^\vee$ and $E^\vee$ 
can be identified with $(M^\vee,V^\vee,F^\vee)$ and $(L^\vee,V^\vee,F^\vee)$
(respectively). Hence we have natural isomorphisms 
$H_{E,D}\cong H_{D,E}^\vee\cong H_{D^\vee,E^\vee}$.
In terms of these isomorphisms we have $F_{E,D}=(F_{D,E}^{-1})^\vee=F_{D^\vee,E^\vee}$, and they induce isomorphisms of $W(k)$-submodules
$O_{E,D}\cong\tilde O_{D,E}^\vee\cong O_{D^\vee,E^\vee}$ because
these $W(k)$-submodules are all defined by the same maximality property. 
Hence we have $\ell_{E,D}=\tilde\ell_{D,E}=\ell_{D^\vee,E^\vee}$.
This proves the lemma as we have $\tilde\ell_{D,E}=\ell_{D,E}$.
\end{proof}

\begin{Lemma}
\label{Le-perm-l}
For each algebraically closed field $\kappa\supseteq k$
we have $\ell_{D,E}=\ell_{D_\kappa,E_\kappa}$.
\end{Lemma}

\begin{proof}
Let $H_\kappa$ and $O_\kappa$ be the analogues of $H$ and $O$ defined with
respect to $D_\kappa$ and $E_\kappa$ instead of $D$ and $E$.
One can check that $H_\kappa=H\otimes_{W(k)}W(\kappa)$ and $O_\kappa=O\otimes_{W(k)}W(\kappa)$. The assertion follows.
\end{proof}

\subsection{The inequality $e_{D,E}\le\ell_{D,E}$}

For $x\in O$ we write $x=x_++x_0+x_{-}$, where $x_+\in O_+$, $x_0\in
O_0$, and $x_{-}\in O_-$. We call $x_+$, $x_0$, $x_{-}$ the components of $x$ in $O$. 

\begin{Lemma}
\label{end0}
Let $x\in O$. The equation in $X$ 
\begin{equation}\label{Eq-in-O}
x=F(X)-X
\end{equation}
has a solution in $O$ which is unique up to the addition by an arbitrary element of $A_0$. If $x\in p^sO$ for some $s\in\NN$,
then there exists a solution in $p^sO$.
\end{Lemma}

\begin{proof}
We define $y_{+}=-\sum_{i=0}^{\infty} F^i(x_+)$ and $y_{-}=\sum_{i=1}^{\infty} F^{-i}(x_{-})$. We have $x_+=F(y_+)-y_+$ and $x_{-}=F(y_{-})-y_{-}$.
Let $\Upsilon_0=(e_{1},\ldots,e_{s})$ be a $\ZZ_p$-basis of $A_0$. Then $\Upsilon_0$ is a $W(k)$-basis of $O_0$, and thus we can
write $x_0=\sum_{i=1}^{s} \gamma_{i} e_{i}$ with $\gamma_{i}\in W(k)$. For $y_{0}=\sum_{i=1}^{s} z_{i} e_{i}$ with $z_{i}\in W(k)$ we have $x_0=F(y_0)-y_0$ if and only if for each $i\in\{1,\ldots,s\}$ we have
$$\sigma(z_{i})-z_{i}=\gamma_{i}.$$
It is well-known that this equation in $z_{i}$ has solutions in $W(k)$. Therefore $y=y_++y_0+y_-$ is a solution of \eqref{Eq-in-O} in $O$. Two solutions of \eqref{Eq-in-O} in $O$ differ by a solution of the equation $F(X)=X$ and thus they differ by an arbitrary element of $A_0$. 
If $x=p^s x'$ with $x'\in O$, then there exists an element $y'\in O$
with $x'=F(y')-y'$, and $p^sy'$ is a solution of \eqref{Eq-in-O} in $p^sO$.\end{proof}

\begin{Lemma}
\label{Le-e-ell}
Let $m\in\NN^*$. Each homomorphism of truncated Dieudonn\'e modules
$$
\zeta_m:(M/p^mM,F,V)\to(L/p^mL,F,V)
$$
can be lifted to a $W(k)$-linear map $\zeta:M\to L$ 
such that $F(\zeta)-\zeta\in p^mH$.
\end{Lemma}

It is not true that every lift of $\zeta_m$ satisfies
$F(\zeta)-\zeta\in p^mH$.

\begin{proof}
By passing to $M\oplus L$ it suffices to consider endomorphisms
instead of homomorphisms, so we consider $\End_{W(k)}(M)$ instead
of $\Hom_{W(k)}(M,L)$. Let $Q=VM$. Then $V$ induces a bijective $\sigma^{-1}$-linear
map $\bar V_m:M/p^mM\to Q/p^mQ$. Let $\zeta_m'=\bar V_m\zeta_m\bar V^{-1}_m:Q/p^mQ\to Q/p^mQ$.
There exists an element $\zeta\in\End_{W(k)}(M)$ which lifts $\zeta_m$ and $\zeta_m'$
at the same time. Indeed, let $M=J\oplus T$ be such that $Q=J\oplus pT$.
Choose $\zeta:T\to M$ to be any lift of the restriction of
$\zeta_m$ to $T/p^m T$ and 
$\zeta:J\to Q$ to be any lift of the restriction
of $\zeta_m'$ to $J/p^mJ$. Then the resulting $W(k)$-linear
map $\zeta:M\to M$
lifts $\zeta_m$ and $\zeta_m'$. It follows that 
$\zeta V-V\zeta$ maps $M$ to $p^m Q$, thus $V^{-1}\zeta V-\zeta$
maps $M$ to $p^mM$, which means that $F(\zeta)-\zeta\in p^m\End_{W(k)}(M)$ as required. 
\end{proof}

\begin{Prop}
\label{Pr-e-ell}
We have $e_{D,E}\leq\ell_{D,E}.$
\end{Prop}

\begin{proof}
Let $m=\ell_{D,E}+1$. We have to show that for each homomorphism 
$$
\zeta_m:(M/p^mM,F,V)\to (L/p^mL,F,V)
$$
of truncated Dieudonn\'e modules, its reduction $\zeta_1:M/pM\to L/pL$ lifts to a homomorphism
of Dieudonn\'e modules $\zeta:(M,F,V)\to (L,F,V)$. By
Lemma \ref{Le-e-ell}, $\zeta_m$ can be lifted to a
$W(k)$-linear map $\zeta':M\to L$ with $F(\zeta')-\zeta'\in p^m H$.
As $p^m H\subseteq pO$, by Lemma \ref{end0} there exists an element
$\xi\in pO$ such that $F(\xi)-\xi=F(\zeta')-\zeta'$. Thus $\zeta=\zeta'-\xi$ has the desired property. 
\end{proof}

\subsection{The inequality $\ell_{D,E}\leq f_{D,E}$}

Let $\mfr_{\cal R}$ be the maximal ideal of a local ring $\cal R$. 

\begin{Lemma}
\label{Le-series}
For $P\in k[[t]]$ and $x\in uk[[u]]$ the expression $P(x)\in k[[u]]$ is well-defined. 
Moreover we have $P(x)=0$ only if either $P=0$ or $x=0$.
\end{Lemma}

\begin{proof}
The power series $P(x)$ converges $u$-adically, so $P(x)$
is well-defined. If $x\neq 0$ and $P\neq 0$, let $r,s\in\NN$
be minimal such that $x\in u^rk[[u]]$ and $P\in t^sk[[t]]$.
Then $r>0$ and $P(x)\in u^{rs}k[[u]]\setminus u^{rs+1}k[[u]]$. Thus $P(x)\neq 0$.
\end{proof}

\begin{Prop}
\label{Pr-ell-f}
We have $\ell_{D,E}\leq f_{D,E}$.
\end{Prop}

\begin{proof}
We show that the assumption that $f_{D,E}<\ell_{D,E}$ leads to a contradiction.
By Lemmas \ref{Le-perm-f} and \ref{Le-perm-l} we can replace the field $k$ by some algebraically closed field extension of it.
Therefore we can assume that there exists an algebraically closed 
field $k^\prime$ and an inclusion of rings $R=k'[[t]]\subset k$
such that $D$ and $E$ are the base change of $p$-divisible groups
$D^\prime$ and $E'$ over $k^\prime$. 

Let $(M^\prime,F^\prime,V^\prime)$ and $(L',F',V')$ be the Dieudonn\'e modules over $k^\prime$ of $D^\prime$ and $E'$ (respectively). We have 
$(M,F,V)=(M^\prime\otimes_{W(k^\prime)} W(k),F^\prime\otimes\sigma,V^\prime\otimes\sigma^{-1})$ 
and similarly for $L$. Let $H'$ and $O^\prime=O^\prime_+\oplus O^\prime_0\oplus O^\prime_{-}$ be the analogues of $H$ and $O=O_+\oplus O_0\oplus O_{-}$ obtained working with $D^\prime$ and $E'$ instead of $D$ and $E$. We have $H=H^\prime\otimes_{W(k^\prime)} W(k)$ and
$O=O^\prime\otimes_{W(k^\prime)} W(k)$.

Let $x\in p^{\ell_{D,E}}H'\cap (O'\setminus pO')$. 
Then $x\in O\setminus pO$.
For each $\eta\in W(\mfr_R)$ we define 
$y_{\eta}=y_{\eta,+}+y_{\eta,0}+y_{\eta,-}\in O$ by
the two $p$-adically convergent series
$$y_{\eta,+}=-\sum_{i=0}^{\infty} F^i(\eta x_+)\in O_+,\qquad
y_{\eta,-}=\sum_{i=1}^{\infty} F^{-i}(\eta x_{-})\in O_-,$$
and by the $t$-adically convergent series in 
$O'\otimes_{W(k')}W(R)\subset O$
$$y_{\eta,0}=-\sum_{i=0}^{\infty} F^i(\eta x_{0}).$$
The last series converges because $\sigma^{p^i}(\eta)\to 0$ for 
$i\to \infty$, so for each $n\in\NN$ the series maps to
a finite sum in $W(R/\mfr_R^n)$. We have
$$\eta x=F(y_{\eta})-y_{\eta}.$$

We note that for every $\eta\in W(k)$ the last equation
can be solved for $y_\eta$ (see Lemma \ref{end0}) and that
the components $y_{\eta,\pm}$ are always given by the
above $p$-adic series, but in general we do not have a natural choice
for $y_{\eta,0}$ and in particular no explicit formula for it.

The reduction modulo $p^{\ell_{D,E}}$ of $y_{\eta}:M\to L$ 
is a homomorphism of truncated Dieu\-donn\'e modules 
$(M/p^{\ell_{D,E}}M,F,V)\to(L/p^{\ell_{D,E}}L,F,V)$.
Indeed, the relation $\eta x=F\circ y_{\eta}\circ F^{-1}-y_{\eta}=V^{-1}\circ y_{\eta}\circ V-y_{\eta}$ gives
$\eta x\circ F=F\circ y_{\eta}-y_{\eta}\circ F$ and 
$V\circ\eta x=-V\circ y_{\eta}+y_{\eta}\circ V$. As $x(M)\subseteq p^{\ell_{D,E}}L$, we conclude that $F\circ y_{\eta}$ and $y_{\eta}\circ F$ (resp. $V\circ y_{\eta}$ and $y_{\eta}\circ V$) coincide modulo $p^{\ell_{D,E}}$. 

By classical Dieudonn\'e theory it follows that each $y_\eta$
defines a homomorphism $D[p^{\ell_{D,E}}]\to E[p^{\ell_{D,E}}]$.
The assumption $f_{D,E}<\ell_{D,E}$ implies that the resulting restrictions
$D[p]\to E[p]$ take only finitely many values, which means that
the reductions of $y_\eta$ in $H/pH$ take only finitely many
values. As the assignment $\eta\mapsto y_\eta$ is additive
and as $\mfr_R$ is infinite it follows that there exists
an element $\eta=(\eta_0,\eta_1,\ldots)\in W(\mfr_R)$
with $\eta_0\neq 0$ such that $y_\eta\in O\cap pH$.

From the relation $y_\eta\in O\cap pH$ we want to deduce the following relations.

\theoremstyle{plain}
\newtheorem*{Claim*}{Claim}

\begin{Claim*} For each $n\in\NN$
the elements $F^n(x_+)$, $F^n(x_0)$, $F^{-n}(x_-)$ lie in $pH$.
\end{Claim*}

\noindent
By the definition of $O$ this implies that $p^{-1}x=p^{-1}(x_++x_0+x_-)\in O$,
which contradicts the assumption that $x\not\in pO$. 
Hence to end the proof of Proposition \ref{Pr-ell-f} 
it remains to prove the claim.

Let $\Vcal=O\cap pH$. We have $pO\subseteq\Vcal\subseteq O$.
Thus $\bar\Vcal=\Vcal/pO$ is a $k$-vector subspace of $O/pO$.
Similarly let $\Vcal'=O'\cap pH'$ and $\bar\Vcal'=\Vcal'/pO'$.
We have $\bar\Vcal=\bar\Vcal'\otimes_{k'}k$.
For an element $b\in O$ (resp.\ $b\in W(k)$)
we denote by $\bar b\in O/pO$ (resp.\ $\bar b\in k$)
its reduction modulo $p$. 
We denote also by $F$ (resp. $F^{-1}$) the $\sigma$-linear (resp. $\sigma^{-1}$-linear) endomorphism of $O_+/pO_+$ and $O_0/pO_0$ (resp. $O_-/pO_-$) induced by $F$ (resp. $F^{-1}$). Let $n_+$ (resp.\ $n_-$) be the minimal non-negative integer such that 
$F^{n_+}(x_+)\in pO_+$ (resp.\ $F^{-n_-}(x_-)\in pO_-$).

An element $\bar z\in O/pO$ lies in $\bar\Vcal$ if and only if 
for every $k'$-linear map $\varpi:O'/pO'\to k'$ with 
$\varpi(\bar\Vcal')=0$ we have $(\varpi\otimes 1_k)(\bar z)=0$.
For $n\in\NN$ we consider the elements of $k'$
$$
a_{+,n}=\varpi (F^n(\bar x_+)),\quad
a_{0,n}=\varpi (F^n(\bar x_0)),\quad
a_{-,n}=\varpi (F^{-n}(\bar x_-)).
$$
For $n\ge n_+$ we have $a_{+,n}=0$ and for $n\ge n_-$ 
we have $a_{-,n}=0$. To prove the claim we have to
show that $a_{+,n}=a_{0,n}=a_{-,n}=0$ for all $n\geq 0$.

As $y_\eta\in\Vcal$ we have $(\varpi\otimes 1_k)(\bar y_\eta)=0$.
From the definition of $y_\eta$ we get
$$
\sum_{n=0}^{n_+}a_{+,n}\eta_0^{p^n}+
\sum_{n=0}^\infty a_{0,n}\eta_0^{p^n}-
\sum_{n=1}^{n_-}a_{-,n}\eta_0^{p^{-n}}=0.
$$
This expression can be viewed as a power series in $R=k'[[t]]$
evaluated at the element $\eta_0^{p^{-n_-}}$ of the
maximal ideal of $R'=k'[[t^{p^{-n_-}}]]$.
By Lemma \ref{Le-series} it follows that we have
$a_{-,n}=0$ for $n\geq 1$ and $a_{+,n}+a_{0,n}=0$
for $n\geq 0$; in particular we have $a_{0,n}=0$ for $n\ge n_+$.
As $F$ is bijective on $O_0/pO_0$, the subspace of
$O_0/pO_0$ generated by $\{F^n(\bar x_0)\mid n\geq 0\}$
is equal to the subspace generated by 
$\{F^n(\bar x_0)\mid n\ge n_+\}$.
Thus we get $a_{0,n}=0$ for $n\geq 0$, which gives also 
$a_{+,n}=0$ for $n\geq 0$. 
As $x\in p^{\ell_{D,E}}H$ and as $0\leq f_{D,E}<\ell_{D,E}$ 
we have $x\in pH$ and hence $x\in\Vcal$. Thus
$0=\varpi(\bar x)=a_{+,0}+a_{0,0}+a_{-,0}$ and therefore
$a_{-,0}=0$.
\end{proof}

\subsection{Conclusions}
\label{Se-conclusions}

\begin{Thm}
\label{Th-f-ell-e}
For two $p$-divisible groups $D$ and $E$ over $k$ we have
$$
f_{D,E}=\ell_{D,E}=e_{D,E}.
$$
\end{Thm}

\begin{proof}
This follows from Propositions \ref{Pr-f-e}, \ref{Pr-e-ell}, and \ref{Pr-ell-f}.
\end{proof}

\begin{Cor}
\label{Co-e-symm}
We have $e_{D,E}=e_{E,D}=e_{D^\vee,E^\vee}$ and similarly for $f$.
\end{Cor}

The equation $e_{D,E}=e_{E,D}$ 
seems not to be obvious from the definitions.

\begin{proof}
This follows from Theorem \ref{Th-f-ell-e} and Lemma \ref{Le-ell-symm}.
\end{proof}

\begin{proof}[Proof of Theorem  \ref{Th-n-ell-e}]
It follows from Theorems \ref{Th-f-ell-e} and \ref{Th-f-n}.
\end{proof}

\begin{Prop}
We have $\ell_{D\oplus E}=\max\{\ell_D,\ell_E,\ell_{D,E}\}$ and $\ell_{D,E}\le n_{D\oplus E}$.
\end{Prop} 

\begin{proof}
We have a direct sum decomposition 
$\End_{W(k)}(M\oplus L)=\End_{W(k)}(M)\oplus \End_{W(k)}(L)\oplus\Hom_{W(k)}(L,M)\oplus\Hom_{W(k)}(M,L)$
 into $W(k)$-modules which is compatible with $F$ and therefore which induces a direct sum decomposition $O_{D\oplus E}=O_D\oplus O_E\oplus O_{D,E}\oplus O_{E,D}$. It follows that $\ell_{D\oplus E}=\max\{\ell_D,\ell_E,\ell_{D,E},\ell_{E,D}\}$. From this and Lemma \ref{Le-ell-symm} we get $\ell_{D\oplus E}=\max\{\ell_D,\ell_E,\ell_{D,E}\}$.
By Theorem \ref{Th-n-ell-e} and the subsequent remark we have
$\ell_{D\oplus E}\leq n_{D\oplus E}$. Thus $\ell_{D,E}\leq n_{D\oplus E}$.
\end{proof}

\begin{Prop}
\label{Pr-min-ell}
If $D$ and $E$ are minimal then $\ell_{D,E}\le 1$ with
equality if and only if both $D$ and $E$ are non-ordinary
minimal.
\end{Prop}

\begin{proof}
This is proved in \cite[Thm.~1.6]{Va3}. 
We give here also a direct argument.
The group $\Hom(D[p],E[p])$ is finite if and only if one of $D$ and
$E$ is ordinary, so $\ell_{D,E}=f_{D,E}$ is zero if and only if one of $D$ and $E$ is ordinary. 
Thus it suffices to prove that $\ell_{D,E}\le 1$
if $D$ and $E$ are minimal. We can assume that $D$ and $E$ are
simple. By Proposition \ref{Pr-isocl-min} 
there are $F$-valuations $w$ on $M^\vee_\QQ$ and $u$ on $L_\QQ$ such that $M^\vee=(M_\QQ^\vee)^{w\ge 0}$ and $L=(L_\QQ)^{u\ge 0}$.
Let $w\otimes u$ be the product valuation on the isoclinic
$\DD$-module $(M^\vee\otimes L)_\QQ=\Hom_{W(k)}(M,L)_\QQ$. 
This is an $F$-valuation. As $M^\vee$ and $L$ have valuative
$W(k)$-bases consisting of elements with valuation in $[0,1)$, 
$M^\vee\otimes L$ has a valuative $W(k)$-basis consisting of
elements with valuations in $[0,2)$. Therefore
$$
(M^\vee\otimes L)_\QQ^{(w\otimes u)\geq 1}
\subseteq\Hom_{W(k)}(M,L)\subseteq
(M^\vee\otimes L)_\QQ^{(w\otimes u)\geq 0}.
$$
As the outer $W(k)$-modules are stable under either $F$ or $F^{-1}$
and as their quotient is annihilated by $p$, we conclude that 
$(M^\vee\otimes L)_\QQ^{(w\otimes u)\geq 1}\subseteq O$ and
$\ell_{D,E}\le 1$.
\end{proof}

\section{Values of homomorphism numbers}
\label{Se-val-hom}

In this section we study possible values of the homomorphism
numbers $e_{D,E}$ and the isomorphism numbers $n_D$; the latter
is a special case of the former by Theorem \ref{Th-n-ell-e}.
In particular, we prove Theorem \ref{Th-n-upper}.
We actually work with $f_{D,E}$ and $\ell_D$, which gives 
equivalent results by Theorems \ref{Th-n-ell-e} and \ref{Th-f-ell-e}.

\subsection{Upper bounds}
\label{Up-b}

We fix $p$-divisible groups $D$ and $D'$ over $k$. The dimension, codimension, height, and Newton polygon
of $D$ (resp.\ $D'$) are denoted $d,c$, $h$, and $\nu$ 
(resp.\ $d',c'$, $h'$, and $\nu'$). 
Recall that Theorem \ref{Th-n-upper} claims
that if $D$ is not ordinary, then the isomorphism number $n_D$ is at most $\lfloor 2\nu(c)\rfloor$. We will provide an upper bound for the coarse homomorphism number $f_{D',D}$, which will imply Theorem \ref{Th-n-upper} by setting $D=D'$.

Let $d^{+}=d+d'$ be the dimension, $c^{+}=c+c'$ be the codimension,
and $h^{+}=h+h'=c^{+}+d^{+}$ be the height of $D\oplus D'$. 
Let $\nu^{+}$ be the Newton polygon of $D\oplus D'$. Let  
$\lambda^{+}_1<\cdots<\lambda^{+}_{r^{+}}$ be the slopes
of $\nu^{+}$.
For each $j^{+}\in\{1,\ldots,r^{+}\}$ let 
$\nu^{+}_{j^{+}}:\RR\to\RR$ be the unique 
linear function of slope $\lambda^{+}_{j^{+}}$ 
such that for all $t\in[0,h^{+}]$ we have
$\nu^{+}(t)=\max\{\nu^{+}_{{j^{+}}}(t)\mid 1\leq j^{+}\leq r^{+}\}$.
Let $\Jcal_D\subseteq\{1,\ldots,r^{+}\}$ be the set of indices
$j^{+}$ such that $\lambda^{+}_{j^{+}}$ is a
slope of $D$.

\begin{Thm}
\label{Th-fDE-upper}
We have $f_{D',D}\leq\max\{\nu^{+}_{j^{+}}(c^{+})\mid 
j^{+}\in \Jcal_D\}$.
\end{Thm}

As $f_{D',D}$ is symmetric (see Corollary \ref{Co-e-symm}), 
by interchanging $D$ and $D'$ we
get another, possibly better upper bound of it. For example, if
either $D$ or $D'$ is ordinary, in this way we get that
$f_{D',D}=0$, which is easily verified directly.
Before proving Theorem \ref{Th-fDE-upper} we deduce some corollaries of it.

\begin{Cor} 
\label{Co-fDE-upper}
The following inequalities hold
$$f_{D',D}\le\nu^{+}(c^{+})\le\nu(c)+\nu'(c')
\le cd/h+c'd'/h'\le c^{+}d^{+}/h^{+}\le h^{+}/4.$$
\end{Cor}

\begin{proof}
This is elementary and left to the reader.
\end{proof}

\begin{Cor}
\label{best}
 We have the following relations
$$n_{D',D}=e_{D',D}=\ell_{D',D}=f_{D',D}\leq \max\{\nu^{+}_{j^{+}}(c^{+})\mid j^{+}\in \Jcal_D\}.$$
\end{Cor}

\begin{proof}
This follows from Remark \ref{Re-f-n}, Theorem \ref{Th-f-ell-e}, and Theorem \ref{Th-fDE-upper}.
\end{proof}

\begin{Cor}
\label{Co-fD-upper}
We have $n_D=f_D\le\lfloor 2\nu(c)\rfloor$.
\end{Cor}

\begin{proof} 
For $D'=D$ we have $r^{+}=r$, $c^{+}=2c$. Thus from Theorem \ref{Th-fDE-upper} we get that $f_D=f_{D,D}\le\max\{\nu^{+}_j(2c)\mid1\le j\le r\}=\nu^{+}(2c)=2\nu(c)$. From this and Theorem \ref{Th-f-n} we get that the corollary holds. 
\end{proof}

\begin{proof}[Proof of Theorem \ref{Th-n-upper}]
It follows from Theorem \ref{Th-f-n} and Corollary \ref{Co-fD-upper}.
\end{proof}

\begin{Lemma}
\label{Le-f-upper-prelim}
To prove Theorem \ref{Th-fDE-upper} we can assume that $D$ and $D'$
are connected with connected duals and that $a_D=a_{D'}=1$.
\end{Lemma}

\begin{proof}
This is similar to the proof of Lemma \ref{Le-b-upper-prelim}.
One has to replace Lemma \ref{Le-prelim} (b) by Lemma \ref{Le-prelim-f} and Theorem \ref{Th-going-down-b} by Theorem \ref{Th-semicont-f} (b). If $D=D'$ one can also use Theorem \ref{Th-going-down-n}
instead of Theorem \ref{Th-semicont-f} (b).
\end{proof}

\begin{proof}[Proof of Theorem \ref{Th-fDE-upper}]
Let $M$ and $M'$ be the Dieudonn\'e modules of $D$ and $D'$ (respectively). By Lemma \ref{Le-f-upper-prelim} we can assume
that $M$ and $M'$ are bi-nilpotent and that $a(M)=a(M')=1$.
Let $z\in M$ and $z'\in M'$ be
generators as $\EE$-modules. Let $M=\EE/\EE\Psi$
and $M'=\EE/\EE\Psi'$ be the associated presentations
given by Lemma \ref{Le-present} (b). For $m\in\NN$ we have
canonical isomorphisms
\begin{multline*}
\Hom(D'[p^m],D[p^m])\cong\Hom_\EE(M'/p^mM',M/p^mM) \\
\cong\Ker(\Psi':M/p^mM\to M/p^mM),
\end{multline*}
where the second isomorphism maps a homomorphism $\phi$ to $\phi(z')$.

Let $m=f_{D',D}$.
By the definition of $f_{D',D}$ there exists an infinite
set $\Lcal$ and a subset $\{x_l\mid l\in\Lcal\}$ of $M$ such that
$\Psi' x_l\in p^mM$ for all $l\in\Lcal$ and such that
the reductions $\bar x_l\in M/pM$
of the elements $x_l$ are pairwise distinct. 

We use Notation \ref{Notat} with respect to $M$,
so that $N=M_\QQ$, the slopes of $N$ are
$\lambda_1<\cdots<\lambda_r$, and for $j\in\{1,\ldots,r\}$ 
we have $\beta_j=\nu_j(c)$. 
For $j\in\{1,\ldots,r\}$ let
$\nu'_j:\RR\to\RR$ be the maximal linear function
of slope $\lambda_j$ such that we have $\nu'_j(t)\leq\nu'(t)$
for all $t\in[0,h']$, and let $\beta'_j=\nu_j'(c')$.
This defines $\beta'=(\beta'_1,\ldots,\beta'_r)\in\RR^r$. We note that $\beta'$ is not 
the analogue of $\beta$ for $M'$ in place of $M$.
By Lemma \ref{Le-Newton} the polygon $\nu'$ coincides with 
$\nu_{\Psi'}$ shifted to the right by $c'$.
Using \eqref{Eq-ww-lambda} this implies
$$
\beta_j'=\nu'_j(c')=\nu_{\Psi',\lambda_j}(0)=\ww_{\lambda_j}(\Psi').
$$

As $N^{\beta+}\subseteq pM$ by Proposition \ref{Pr-fil},
the images of $x_l$ in $N^{\underline 0}/N^{\beta+}$
are all distinct. By Lemma \ref{Le-Phi} the operator
$\Psi'$ induces a homomorphism 
$N^{\underline 0}/N^{\beta+}\to N^{\beta'}/N^{(\beta+\beta')+}$
with finite kernel. Hence $\Psi'x_l\not\in N^{(\beta+\beta')+}$
for all but finitely many $l\in\Lcal$. 
Recall that $\underline m=(m,\ldots,m)\in\RR^r$. 
As $\Psi' x_l\in p^mM$
and $p^mM\subseteq N^{\underline m}$, it follows that 
$N^{\underline m}\not\subseteq N^{(\beta+\beta')+}$. 
Hence for some index $j\in\{1,\ldots,r\}$ we have 
$m\leq\beta_j+\beta'_j$.
There exists a unique $j^{+}\in \Jcal_D\subseteq\{0,\ldots,r^{+}\}$
such that $\lambda_j=\lambda^{+}_{j^{+}}$.
By Lemma \ref{Le-calculate} below we have 
$\beta_j+\beta'_j=\nu^{+}_{j^{+}}(c^{+})$ and thus $f_{D',D}=m\le\nu^{+}_{j^{+}}(c^{+})$.
\end{proof}

\begin{Lemma}
\label{Le-calculate}
With the above notations, we have 
$\nu_j(c)+\nu'_j(c')=\nu^{+}_{j^{+}}(c^{+})$.
\end{Lemma}

\begin{proof}
Let $x^{+}\in[0,h^{+}]$, $x\in[0,h]$, and $x'\in[0,h']$
be the maximal elements such that $\nu^{+}$, $\nu$, and $\nu'$
have slope $<\lambda_j$ on the intervals $[0,x^{+}]$,
$[0,x]$, and $[0,x']$ (respectively). 
Using that $x^{+}=x+x'$, $c^{+}=c+c'$, 
and $\lambda^+_{j^+}=\lambda_j$, we compute
$\nu^{+}_{j^{+}}(c^{+})-\nu_j(c)-\nu'_j(c')=
\nu^{+}_{j^{+}}(x^{+})-\nu_j(x)-\nu'_j(x')=
\nu^{+}(x^{+})-\nu(x)-\nu'(x')=0.$
\end{proof}

\begin{Remark}
Corollary \ref{Co-fD-upper} can be proved directly by
letting $D=D'$ in the proof of Theorem \ref{Th-fDE-upper}. 
Then the last two lines of that proof can be replaced
by the conclusion
$f_D\le\beta_j+\beta'_j=2\beta_j\le 2\nu(c)$,
which avoids Lemma \ref{Le-calculate}.
\end{Remark}

Homomorphism numbers have bounded variation under isogenies:

\begin{Prop}
\label{Pr-f-change} Let $E$ denote a $p$-divisible group, and
let $g:D\to D'$ be an isogeny of $p$-divisible
groups over $k$ such that $p$ annihilates the
kernel of $g$. Then we have $|f_{D,E}-f_{D',E}|\leq 1$
and $|f_D-f_{D'}|\le 2$, thus $|n_D-n_{D'}|\le 2$
by Theorem \ref{Th-f-n}.
\end{Prop}

\begin{proof}
Let $g':D'\to D$ be the isogeny such that $g'g=p \cdot 1_D$.
We write $H_m=\Hom(D[p^m],E[p^m])$ and 
$H'_m=\Hom(D'[p^m],E[p^m])$. For each $m\in\NN$ 
we have a commutative diagram
$$
\xymatrix@M+0.2em@C+2em{
H_{m+2} \ar[d]_{u\mapsto ug'}
\ar[rr]^{\tau_{m+2,1}} &&
H_1 \ar[d]^\rho \\
H'_{m+2} \ar[r]^-{\tau_{m+2,2}} &
H'_2 \ar[r]^-{u\mapsto ug} &
H_2
}
$$
where $\rho$ maps a homomorphism $u:D[p]\to E[p]$ to the obvious
composition $D[p^{2}]\to D[p]\xrightarrow uE[p]\to E[p^{2}]$.
Let $m=f_{D',E}$. Then the image of $\tau_{m+2,2}$ is finite 
by the definition of $f_{D',E}$.
As $\rho$ is injective, the image of $\tau_{m+2,1}$ is
finite too, thus $f_{D,E} \leq m + 1$. By interchanging
the roles of $D$ and $D'$ we get that
$|f_{D,E}-f_{D',E}|\leq 1$.
A similar argument (or Corollary \ref{Co-e-symm}) gives
also $|f_{E,D}-f_{E,D'}|\leq 1$, thus
$|f_D-f_{D'}|\leq|f_D-f_{D,D'}|+|f_{D,D'}-f_{D'}|\le 2$.
\end{proof}

\smallskip
\begin{Remark}
\label{Re-n-upper}
The preceding proposition offers another approach to
bound $f_D$ from above; see Remark \ref{Re-b-bound}.
For a non-ordinary minimal $p$-divisible group $D_0$ 
we have $n_{D_0}=f_{D_0}=e_{D_0}=\ell_{D_0}=1$;
see \cite[Thm.~1.2]{Oo3} and \cite[Thm.~1.6]{Va3}. 
Here $\ell_{D_0}=1$ is proved in Proposition \ref{Pr-min-ell}, 
and the rest follows 
by Theorems \ref{Th-f-ell-e} and \ref{Th-f-n}.
Thus Proposition \ref{Pr-f-change} 
(or \cite[Prop.~1.4.4]{Va3}) gives $f_D\le 1+2q_D$. Together with
Theorem \ref{Th-min-isog} we get $f_D\leq 1+2\lfloor\nu(c)\rfloor$.
As $\lfloor 2\nu(c)\rfloor\leq 1+2\lfloor\nu(c)\rfloor$
with equality if and only if $\lfloor 2\nu(c)\rfloor$ is odd,
this is slightly weaker than Corollary \ref{Co-fD-upper},
but we get as close as possible taking into account
that this approach necessarily gives an odd upper bound.
\end{Remark}

\subsection{Lower bounds in the isoclinic case}
\label{Subsec-n-isocl}

We assume now that the $p$-divisible group $D$ over
$k$ is isoclinic and continue to study the 
possible values of $f_D=\ell_D$ in this case.

Let us recall from \cite{Va3} how to compute $\ell_D$
for isoclinic groups.
Let $M$ be the Dieudonn\'e module of $D$ and let 
$O\subseteq\End_{W(k)} (M)$ be its level module. 
We have $O=O_0$, cf.\ Definition \ref{Def-ell}.
For each integer $t$ let $\alpha_t(M)$
be the maximal integer and let $\beta_t(M)$ be the minimal
integer such that
$$
p^{\beta_t(M)}M\subseteq F^tM\subseteq p^{\alpha_t(M)}M.
$$
Let $\delta_t(M)=\beta_t(M)-\alpha_t(M)$.

\begin{Lemma}
\label{Le-delta-t}
The natural number $\delta_t(M)$ is equal to the minimal integer 
$\delta$ such that $p^\delta\End_{W(k)}( M)\subseteq\End_{W(k)}(F^tM)$ 
as subgroups of $\End_{W(k)}(M)_\QQ$.
\end{Lemma}

\begin{proof}
The optimality of $\alpha_t(M)$ and $\beta_t(M)$ means
that the inclusion maps $p^{\beta_t(M)}M\to F^tM$ and
$p^{-\alpha_t(M)}M^\vee\to (F^tM)^\vee$ have non-zero
reductions modulo $p$. This property carries
over to the tensor product of the two
maps, which is an inclusion 
$p^{\beta_t(M)-\alpha_t(M)}\End_{W(k)}(M)\to\End_{W(k)}(F^tM)$.
\end{proof}

\begin{Lemma}
[{\cite[Prop.~4.3 (a)]{Va3}}]
\label{Le-ell-isocl}
We have $\ell_D=\max\{\delta_t(M)\mid t\in\NN\}$.
\end{Lemma}

\begin{proof}
The level module $O=O_0$ is the intersection in
$\End_{W(k)}(M)_\QQ$ of all $\End_{W(k)}(F^tM)$ for $t\in\NN$. 
Thus $\ell_D$ is minimal with 
$p^{\ell_D}\End_{W(k)}(M)\subseteq\End_{W(k)}(F^tM)$ for all $t\in\NN$.
The lemma follows from Lemma \ref{Le-delta-t}.
\end{proof}

\begin{Lemma} 
\label{Le-ell-isocl-a}
Assume that $a_D=1$.

\smallskip
(i) For $0\le t\le c$ we have $\alpha_t(M)=0$.

\smallskip
(ii) For $0\le t\le d$ we have $\beta_t(M)=t$.
\end{Lemma}

\begin{proof}
We write $M=\EE/\EE\Psi$ as in Lemma \ref{Le-present} (c).
Let $z=1+\EE\Psi\in M$.
For $0\le t\le c$ we have $F^tM\subseteq M$,
and $F^tM\not\subseteq pM$ because $F^tz$ is part 
of the $W(k)$-basis $\Upsilon$ of $M$ defined in 
Lemma \ref{Le-present} (a). 
This proves (i). Similarly, for $0\le t\le d$ we have
$V^tM\subseteq M$ and $V^tM\not\subseteq pM$,
which is equivalent to $p^tM\subseteq F^tM$
and $p^tM\not\subseteq pF^tM$. This proves (ii).
\end{proof}

We have the following lower bound of $\ell_D$.

\begin{Prop}
\label{Pr-ell-lower}
If $D$ is isoclinic with $a_D=1$ then $\ell_D\ge\min\{c,d\}$.
\end{Prop}

\begin{proof}
Let $t=\min\{c,d\}$. 
We calculate $\ell_D\ge\delta_t(M)=\beta_t(M)-\alpha_t(M)=t$
using Lemmas \ref{Le-ell-isocl} and \ref{Le-ell-isocl-a}.
\end{proof}

\begin{Cor}
\label{Co-ell-lower}
Assume that $D$ is isoclinic with $a_D=1$ and $|c-d|\le 2$.
Then we have $\ell_D=\min\{c,d\}$.
\end{Cor}

\begin{proof}
We have $\min\{c,d\}\le\ell_D=f_D\le
\lfloor 2\nu(c)\rfloor\le\lfloor 2cd/(c+d)\rfloor=\min\{c,d\}$, cf. Theorem \ref{Th-f-ell-e} and Corollary \ref{Co-fD-upper}.
\end{proof}

The lower bound in Proposition \ref{Pr-ell-lower} is optimal:

\begin{Example}
\label{EX3}
Let $D$ be the isoclinic $p$-divisible group with
Dieudonn\'e module $M=\EE/\EE\Psi$ for $\Psi=F^c+V^d$
where $cd>0$. Then $\ell_D=\min\{c,d\}$.
\end{Example}

\begin{proof}
This is a particular case of \cite[Thm.~1.5.2]{Va3}.
We give here a direct proof. 
As the $W(k)$-basis 
$\Upsilon$ of $M$ defined in Lemma \ref{Le-present} (a)
is annihilated by $F^{c+d}+p^d$
we have $F^{c+d}M=p^dM$. It follows that
$\alpha_{t+c+d}(M)=\alpha_t(M)+d$ and
$\beta_{t+c+d}(M)=\beta_t(M)+d$ for $t\in\ZZ$.
We also have $\alpha_{-t}(M)=-\beta_{t}(M)$.
Thus Lemma \ref{Le-ell-isocl-a} gives a complete
description of $\delta_t(M)$ for all $t$,
which implies that $\delta_t(M)\le\min\{c,d\}$
with equality when $t=\min\{c,d\}$.
Therefore $\ell_D=\min\{c,d\}$ by Lemma \ref{Le-ell-isocl}.
\end{proof}

We recall that Corollary \ref{Co-n-upper} is a direct consequence of Theorem \ref{Th-n-upper} and of the following optimality result.

\begin{Prop}
\label{Pr-n-isocl}
Assume that $\nu$ is linear and $cd>0$.
Then there exists a $p$-divisible group $D$ over $k$
with Newton polygon $\nu$ such that
$\ell_D=\lfloor 2\nu(c)\rfloor$.
(Note that $n_D=e_D=f_D=\ell_D$ by Theorems \ref{Th-f-n}
and \ref{Th-f-ell-e}.)
\end{Prop}

\begin{proof}
By passing to the dual if necessary we can assume that
$c\geq d$. We have $\nu(c)=cd/(c+d)$. Let $m=\lfloor 2\nu(c)\rfloor$.
We will construct explicitly the Dieudonn\'e module of $D$
to be $M=\EE/\EE\Psi$ for a suitable $\Psi\in\EE$ as in
Lemma \ref{Le-present} (b). Let $z=1+E\Psi\in M$. 
By Corollary \ref{Co-fD-upper} and Theorem \ref{Th-f-ell-e},
in order that $\ell_D=m$ it suffices to show that $\ell_D\geq m$.

We will show that $\Psi$ can be chosen such that $M$ is
isoclinic and there exists $x\in M\setminus pM$ with $F^cx=p^mz$. 
This implies that $\beta_c(M)\ge m$.
By Lemmas \ref{Le-ell-isocl} and \ref{Le-ell-isocl-a} we
get $\ell_D\ge\delta_c(M)=\beta_c(M)\ge m$ as required.

It remains to construct $\Psi$ and $x$.
If $c=d$, then $m=c=d$ and we can take $\Psi=F^c+V^c$ and
$x=V^cz$ (or use Corollary \ref{Co-ell-lower}).
If $c>d$, then we can take $\Psi=F^c+\Phi+V^d$ for
$$
\Phi=p^{2d-m}F^{c-2d}=p^{c-m}V^{2d-c}.
$$
A priori $\Phi$ is an element of $\DD=\EE\otimes\QQ$,
but actually $\Phi$ lies in $\EE$ (as $m<2d$ and $m<c$).
We take $x=-(p^{m-d}F^d+V^d)z$. 
It is easy to see that as $\Psi z=0$ we have
$F^cx=p^mz$. As $m\ge d$ and $c>d$, we have $x\in M\setminus pM$.
The Newton polygon of $M$ is linear of slope $\lambda=d/(c+d)$.
Indeed, as the exponents of $F$ in $\Psi$ satisfy
$c>c-2d>-d$, this is equivalent to $2d-m\geq 2d\lambda$
(cf.\ Lemma \ref{Le-Newton}) and thus $m\le 2cd/(c+d)$
which obviously holds.
\end{proof}

\begin{Example}
More generally, for each integer $m$ with
$\min\{c,d\}\leq m\leq\lfloor 2cd/(c+d)\rfloor$
there exists an isoclinic $p$-divisible group $D$
with $\ell_D=m$. Indeed, we can assume that $c\geq d$
and that $m<\lfloor 2cd/(c+d)\rfloor$.
Define $D$ such that its Dieudonn\'e module is $\EE/\EE\Psi$
with $\Psi=F^c+p^{2d-m}F^{c-2d}+V^d$.
Then the proof of Proposition \ref{Pr-n-isocl}
gives $\ell_D\ge m$. 

We sketch a proof of the other 
inequality $\ell_D=f_D\le m$ using the method of the proof
of Theorem \ref{Th-fDE-upper}.
Let $w$ be the minimal $F$-valuation of slope 
$\lambda=d/(c+d)$ on $N=M_\QQ$ such that
$w(M)\ge 0$. Let $\beta=cd/(c+d)$.
Let $\psi:M\to M/p^{m+1}M$ be induced by $\Psi$.
We say that a subgroup of $M$ has infinite 
reduction if its image in $M/pM$ is infinite.

Assume that $f_D>m$. This means that $\Ker\psi$
has infinite reduction. 
As $p^{m+1}M\subseteq N^{w\ge m+1}$ and as 
$\ww_\lambda(\Psi)=\beta$, using Lemma \ref{Le-Phi}
it follows that the kernel of the restriction
$\psi':M\cap N^{w\ge m+1-\beta}\to M/p^{m+1}M$
has infinite reduction as well. Thus the
kernel of the induced map $\psi'':M\cap N^{w\ge m+1-\beta}
\to M/(p^{m+1}M+N^{w>2\beta})$ has infinite reduction.
If we identify $M$ with $W(k)^{c+d}$ using the basis 
$\Upsilon$ defined in Lemma \ref{Le-present} (a), then
the map $\psi''$ defined for $\Psi$ is
equal to the analogous map $\psi''_0$ defined for 
$\Psi_0=F^c+V^d$. This is true because 
$\ww_\lambda(\Phi)+m+1-\beta>2\beta$, 
which holds as $\ww_\lambda(\Phi)=3\beta-m$.

We continue with $M_0=\EE/\EE\Psi_0$.
Let $z=1+\EE\Psi_0\in M_0$. 
As before let $w$ be the minimal $F$-valuation
of slope $\lambda$ on $N_0=(M_0)_\QQ$ such that
$w(M_0)\ge 0$.
Let $M_1=\Psi_0(M_0)$ and $M_2=p^{d+1}M_0+N_0^{w>2\beta}$.
It is easy to see that $M_2\subseteq pM_1$, using that
the elements $(p^iF^{c-i}z)_{0\le i\le d}$,
$(p^dF^jz)_{0<j<c-d}$, and $(p^iV^{d-i}z)_{0<i\le d}$
form a basis of $M_1$.
By the previous paragraph, the kernel of the 
map $\psi_0:M_0\to M_0/M_2$ induced by $\Psi_0$
has infinite reduction.
Thus the kernel of $\psi_0:M_0\to M_0/pM_1$ has infinite
reduction, and the kernel of $\psi_0:M_0/pM_0\to M_1/pM_1$
is infinite. This is impossible.

Therefore we have $\ell_D=f_D\le m$.
\end{Example}

\subsection{Possible values of isomorphism numbers}

For a non-ordinary Newton polygon $\nu$
let $\scrN_\nu$ be the set of all possible values of 
$n_D=e_D=f_D=\ell_D$ for
$p$-divisible groups $D$ with Newton polygon $\nu$. We make two
fragmentary remarks on the structure of $\scrN_\nu$. First, as
two isogenous groups can be linked by a chain of isogenies with
kernels annihilated by $p$, in view of Proposition
\ref{Pr-f-change} the difference between two consecutive numbers in $\scrN_\nu$ is at most $2$.
In certain cases we can say more (cf.\ Proposition \ref{Pr-b-values}):

\begin{Prop}
Assume that $\lfloor 2\nu(c)\rfloor$ is odd and lies in
$\scrN_\nu$.
Then $\scrN_\nu$ contains all odd integers between $1$
and $\lfloor 2\nu(c)\rfloor$.
\end{Prop}

\begin{proof}
Let $m=\lfloor\nu(c)\rfloor$. We choose $D$ such that
$f_D=\lfloor 2\nu(c)\rfloor=2m+1$.
By Theorem \ref{Th-min-isog} there exists a chain of isogenies
$D=D_m\to \cdots\to D_0$
such that all consecutive kernels are annihilated
by $p$ and $D_0$ is minimal. As $f_{D_0}=1$ (resp.\ $f_D=2m+1$),
for each $i\in\{0,1,\ldots,m\}$ we get from Proposition \ref{Pr-f-change} that $f_{D_i}\le 1+2i$ (resp. $f_{D_i}\ge 1+2i$). Thus 
$f_{D_i}=1+2i$.
\end{proof}

\subsection{The principally quasi-polarised case}
\label{Se-p.q.p.}

Let $D$ be a $p$-divisible group over $k$ of dimension $d$
equipped with a 
principal quasi-polarisa\-tion $\lambda:D\to D^\vee$; thus $c=d>0$. 
The isomorphism number $n_{D,\lambda}$ of $(D,\lambda)$ 
is the least level $m$ such that $(D[p^m],\lambda[p^m])$ 
determines $(D,\lambda)$ up to isomorphism. 
It is proved in \cite[Subsect.~6.3]{GV} that $n_{D,\lambda}\le n_D$ if $p>2$ and $n_{D,\lambda}\le n_D+1$ if $p=2$. If $D$ is not supersingular, then Theorem \ref{Th-n-upper} implies that
$n_D\le d-1$. If $D$ is supersingular and $d>0$, then we have
$n_{D,\lambda}\le d$ by \cite[Thm.~1.3]{NV1}.
Together we get in all cases:

\begin{Cor}\label{Co-p.q.p.}
We have $n_{D,\lambda}\le d$. 
\end{Cor}

\noindent
This bound is optimal by \cite[Example 3.3]{NV1}.

\subsection{The number $N_h$}

\noindent
Recall that $N_h$ is defined in Section \ref{Se-prelim} as the minimal number such that for all $p$-divisible groups $D$ and $E$ over $k$ of height at most $h$, we have $e_{D,E} \leq N_h$.

\begin{Prop} 
\label{Pr-Nh}
We have $N_h=\lfloor {h/2}\rfloor$.
\end{Prop}

\begin{proof}
By Theorem \ref{Th-f-ell-e} and Corollary \ref{Co-fDE-upper} we have $N_h\leq {h/ 2}$.
By Example \ref{EX3} there exists an isoclinic $p$-divisible group $D$ over $k$ of slope $1/2$ 
and dimension $\lfloor {h/ 2}\rfloor$ with $f_D=\lfloor {h/ 2}\rfloor$; its height is $2\lfloor {h/ 2}\rfloor\le h$. As we have $e_{D,D}=
f_D$ by Theorem \ref{Th-f-ell-e}, we get that $N_h\geq \lfloor {h/ 2}\rfloor$. We conclude that $N_h=\lfloor {h/ 2}\rfloor$.
\end{proof}

\begin{Remark} 
In the last proof the reference to Example \ref{EX3} can be replaced
by \cite[Example 3.3]{NV1} if we use that $f_D=n_D$.
\end{Remark}

\bigskip\noindent
{\bf Acknowledgments.} The second author was supported by a
postdoctoral research fellowship of the FQRNT while enjoying the
hospitality of the Institut de math\'ematiques de Jussieu
(Universit\'e Paris $7$). He also thanks T. It\=o for arranging
good work conditions at the University of Ky\=oto in December
$2008$. The third author would like to thank Binghamton University
and Tata Institute for Fundamental Research, Mumbai for good
work conditions and Offer Gabber for helpful discussions. 
He was partially supported by the NSF grant DMS \#0900967.
\bigskip


\bigskip\bigskip

\hbox{Eike Lau}
\hbox{Email: elau@math.upb.de}

\hbox{Address: Universit\"at Paderborn, Institut f\"ur Mathematik, }
\hbox{Warburger Str.~100, 33098 Paderborn, Germany.}

\bigskip
\hbox{Marc--Hubert Nicole}
\hbox{Email: nicole@iml.univ-mrs.fr} 
\hbox{Address: Institut math\'ematique de Luminy,} 
\hbox{Campus de Luminy, Case 907, 
13288 Marseille cedex 9, 
France.}

\bigskip
\hbox{Adrian Vasiu}
\hbox{Email: adrian@math.binghamton.edu}

\hbox{Address: Department of Mathematical Sciences,
Binghamton University,} 
\hbox{Binghamton, P. O. Box 6000, NY 13902-6000, U.S.A.}

\begin{thebibliography}{NVW}

\bibitem[Be]{Be}
P.~Berthelot, Th\'eorie de Dieudonn\'e sur un anneau 
de valuation parfait, {\em Ann.\ Sci.\ \'Ecole\ Norm.\ Sup.}\
(4) {\bf 13} (1980), 225--268

\bibitem[dJO]{dJO}
J.~de Jong and F.~Oort, Purity of the stratification by Newton polygons,
{\em J.\ Amer.\ Math.\ Soc.} {\bf 13} (2000), no.~1, 209--241

 \bibitem[De]{De} M.~Demazure, {\em Lectures on $p$-divisible groups}, Lecture Notes in Math., Vol.{\bf~302}, Springer-Verlag, Berlin-New York, 1972

\bibitem[GV]{GV}
O.~Gabber and A.~Vasiu, Dimensions of group schemes of
automorphisms of truncated Barsotti--Tate groups, to appear in {\em Int. Math. Res. Notices}, 49 pages published online, http://imrn.oxfordjournals.org/content/early/2012/07/26/
imrn.rns165.abstract?keytype=ref\&ijkey=gVkQc5Otk95hU2C

\bibitem[Ill]{Ill} L.~Illusie, D\'eformations de groupes de Barsotti--Tate (d'apr\`es A. Grothendieck), in {\em S\'eminaire sur les pinceaux arithm\'etiques: la conjecture
de Mordell} (Paris, 1983/84), {\em Ast\'erisque}, No.\ {\bf 127}
(1985), 151--198

\bibitem[Ka]{Ka}
N.~Katz, Slope filtration of $F$-crystals,
{\em Journ\'ees de G\'eom\'etrie Alg\'ebrique de Rennes}
(Rennes, 1978), Vol.~I, {\em Ast\'erisque}, No.\ \textbf{63} (1979), 113--163

\bibitem[La]{La} E.~Lau,
Displays and formal $p$-divisible groups,
{\em Invent.\ Math.} \textbf{171} (2008), no.~3, 617--628

\bibitem[Ma]{Ma} Y.~I.~Manin, The theory of formal
commutative groups in finite characteristic, {\em Russian Math.\
Surv.} \textbf{18} (1963), no.~6, 1--83

\bibitem[NV1]{NV1} M.-H. Nicole and A.~Vasiu, Minimal truncations of supersingular $p$-divisible groups, {\em Indiana Univ. Math. J.} {\bf 56} (2007), no.~6,  2887--2897

\bibitem[NV2]{NV2} M.-H.~Nicole and A.~Vasiu, Traverso's isogeny conjecture for $p$-divisible groups, {\em Rend. Semin. Mat. Univ. Padova} {\bf 118} (2007), 73--83

\bibitem[NVW]{NVW} M.-H. Nicole, A.~Vasiu, and T.~Wedhorn,
Purity of level $m$ stratifications, {\em Ann.\ Sci.\ \'Ecole Norm.\ Sup.} (4) {\bf 43} (2010), no.~6, 925--955

\bibitem[Oo1]{Oo1} F.~Oort,
Newton polygon strata in the moduli of abelian varieties,
{\em Moduli of Abelian Varieties} (Texel Island, 1999), 417--440,
Progr.\ Math., {\bf 195}, Birkh\"auser, Basel, 2001

\bibitem[Oo2]{Oo2} F.~Oort,
Foliations in moduli spaces of abelian varieties,
{\em J.\ Amer.\ Math.\ Soc.} \textbf{17} (2004), no.~2, 267--296

\bibitem[Oo3]{Oo3} F.~Oort, Minimal $p$-divisible groups,
{\em Ann.\ of Math.} (2) \textbf{161} (2005), no.~2, 1021--1036

\bibitem[Oo4]{Oo4} F.~Oort, Simple $p$-kernels of $p$-divisible groups,  {\em Adv.\ Math.}  {\bf 198} (2005), no.~1, 275--310

\bibitem[OZ]{OZ} F.~Oort and Th.~Zink,
Families of $p$-divisible groups with constant Newton polygon,
{\em Doc.\ Math.} \textbf{7} (2002), 183--201

\bibitem[RZ]{RZ} M.~Rapoport and Th.~Zink,
{\em Period spaces for $p$-divisible groups}, Annals of
Mathematics Studies, Vol.  \textbf{141}, Princeton University Press, Princeton, NJ, 1996

\bibitem[Tr1]{Tr1} C.~Traverso,
Sulla classificazione dei gruppi analitici commutativi di
caratteristica positiva, \emph{Ann.\ Scuola Norm.\ Sup.\ Pisa} (3)
\textbf{23} (1969), 481--507

\bibitem[Tr2]{Tr2} C.~Traverso, $p$-divisible groups over fields, \emph{Symposia Mathematica}, Vol. {\bf XI} (Convegno di Algebra Commutativa, INDAM,
Rome, 1971), 45--65, Academic Press, London, 1973

\bibitem[Tr3]{Tr3} C.~Traverso, Specializations of Barsotti--Tate groups, {\em Symposia Mathematica}, Vol. {\bf XXIV} (Sympos., INDAM, Rome, 1979),  1--21, Academic Press, London-New York, 1981

\bibitem[Va1]{Va1} A.~Vasiu, Crystalline boundedness principle,
{\em Ann. Sci. \'Ecole Norm. Sup.} {\bf 39} (2006), no.~2, 245--300

\bibitem[Va2]{Va2}
A.~Vasiu, Level $m$ stratifications of versal deformations of $p$-divisible groups,
{\em J.\ Algebraic Geom.}  {\bf 17}  (2008),  no.~4, 599--641

\bibitem[Va3]{Va3} A.~Vasiu, Reconstructing $p$-divisible groups from their truncations of small level, 
{\em Comment.\ Math.\ Helv.} {\bf 85} (2010), no.~1, 165--202

\bibitem[Yu]{Yu} C.-F.~Yu, On finiteness of endomorphisms of abelian
varieties, {\em Math. Res. Lett.} {\bf 17} (2010), no. 2, 357--370

\bibitem[Zi1]{Zi1}
Th.~Zink, The display of a formal $p$-divisible group,
{\em Cohomologies $p$-adiques et applications arithm\'etiques}, Vol.~I,
{\em Ast\'erisque}, No.\ {\bf 278} (2002), 127--248

\bibitem[Zi2]{Zi2}
Th.~Zink, de Jong-Oort purity for $p$-divisible groups, {\it Algebra, arithmetic, and geometry: in honor of Yu.\ I.\ Manin}, Vol. II, 693--701, Progr.\ Math., {\bf 270}, Birkh\"auser Boston, Inc., Boston, MA, 2009

\end{thebibliography}
\end{document}